\documentclass[12pt]{amsart}
\newcommand{\preprint}[1]{#1}
\newcommand{\journal}[1]{}
\newcommand{\hide}[1]{}

\usepackage{amssymb}
\usepackage{amsbsy}
\usepackage{amscd}
\usepackage{amsmath}
\usepackage{amsthm}

\theoremstyle{plain}
\newtheorem{thm}{Theorem}[section]
\newtheorem{prop}[thm]{Proposition}
\newtheorem{claim}[thm]{Claim}

\newtheorem{cor}[thm]{Corollary}
\newtheorem{lem}[thm]{Lemma}

\newtheorem{assumption}[thm]{Assumption}

\theoremstyle{definition}
\newtheorem{defi}[thm]{Definition}

\theoremstyle{remark}
\newtheorem{example}[thm]{Example}
\newtheorem{question}[thm]{Question}
\newtheorem{rem}[thm]{Remark}

\newcommand{\Def}{{\mathcal D}ef}

\renewcommand{\AA}{{\mathbb A}}
\newcommand{\B}{{\mathcal B}}

\newcommand{\E}{{\mathcal E}}
\newcommand{\F}{{\mathcal F}}

\newcommand{\Hun}{\H}

\newcommand{\M}{{\mathcal M}}

\newcommand{\PP}{{\mathbb P}}
\newcommand{\UU}{{\mathbb U}}
\newcommand{\Q}{{\mathcal Q}}

\newcommand{\TT}{{\mathbb T}}

\newcommand{\X}{{\mathcal X}}
\newcommand{\Y}{{\mathcal Y}}
\newcommand{\Z}{{\mathcal Z}}
\newcommand{\RealNumbers}{{\mathbb R}}
\newcommand{\Integers}{{\mathbb Z}}
\newcommand{\ComplexNumbers}{{\mathbb C}}
\newcommand{\RationalNumbers}{{\mathbb Q}}

\newcommand{\IsomRightArrowOf}[1]{
\stackrel
{\stackrel{#1}{\cong}}
{\rightarrow}
}

\newcommand{\LongIsomRightArrow}{\stackrel{\cong}{\longrightarrow}}
\newcommand{\RightArrowOf}[1]{\stackrel{#1}{\rightarrow}}

\newcommand{\LongLeftArrowOf}[1]{\stackrel{#1}{\longleftarrow}}

\newcommand{\LongRightArrowOf}[1]{\stackrel{#1}{\longrightarrow}}
\newcommand{\LongIsomRightArrowOf}[1]{
\stackrel
{\stackrel{#1}{\cong}}
{\longrightarrow}
}

\newcommand{\StructureSheaf}[1]{{\mathcal O}_{#1}}
\newcommand{\EndProof}{\hfill  $\Box$}
\newcommand{\restricted}[2]{#1_{\mid_{#2}}}

\newcommand{\rank}{{\rm rank}}

\newcommand{\Pic}{{\rm Pic}}
\newcommand{\Sym}{{\rm Sym}}
\newcommand{\Ext}{{\rm Ext}}

\newcommand{\Hom}{{\rm Hom}}
\newcommand{\Aut}{{\rm Aut}}

\newcommand{\Abs}[1]{|\!#1\!|}

\newcommand{\DDef}{{\mathbb D}ef}

\begin{document}
\title[Modular Galois covers]
{Modular Galois covers associated to symplectic resolutions of singularities}
\author{Eyal Markman}
\address{Department of Mathematics and Statistics, 
University of Massachusetts, Amherst, MA 01003}
\email{markman@math.umass.edu}

\begin{abstract}
Let $Y$ be a normal projective variety and $\pi:X\rightarrow Y$
a projective holomorphic symplectic resolution. 
Namikawa proved that the Kuranishi deformation spaces
$Def(X)$ and $Def(Y)$ are both smooth, of the same dimension, 
and $\pi$ induces a finite branched cover $f:Def(X)\rightarrow Def(Y)$.
We prove that $f$ is Galois. 
We proceed to calculate the Galois group $G$, 
when $X$ is simply connected, and its holomorphic symplectic structure
is unique, up to a scalar factor.
The singularity of $Y$ is generically of $ADE$-type, 
along every codimension $2$ irreducible component $B$ of the singular locus, 
by Namikawa's work. The
modular Galois group $G$ is the product of Weyl groups of finite type, indexed
by such irreducible components $B$. 
Each Weyl group factor $W_B$ is that of a Dynkin diagram, 
obtained as a quotient of the Dynkin diagram of the singularity-type of $B$, 
by a group of Dynkin diagram automorphisms. 

Finally we consider generalizations of the above set-up, where $Y$
is affine symplectic, or a Calabi-Yau threefold with a curve of
$ADE$-singularities. We prove that $f:Def(X)\rightarrow Def(Y)$
is a Galois cover of its image.
This explains the
analogy between the above results and related work of Nakajima, on 
quiver varieties, and of Szendr\Hun{o}i on enhanced gauge symmetries for
Calabi-Yau threefolds.
\end{abstract}

\maketitle

\tableofcontents 
\section{Introduction}
An {\em irreducible holomorphic symplectic manifold} is a simply connected
compact K\"{a}hler manifold $X$, such that $H^0(X,\Omega^2_X)$ is generated
by an everywhere non-degenerate holomorphic two-form 
\cite{beauville,huybrects-basic-results}.
A projective {\em symplectic variety}
is a normal projective variety with rational Gorenstein singularities, 
which regular locus admits a non-degenerate holomorphic two-form.
Let $\pi:X\rightarrow Y$ be a symplectic resolution of a 
normal projective variety $Y$. Then $Y$ has rational Gorenstein singularities
and is thus a symplectic variety
\cite{beuville-symplectic-singularities}, Proposition 1.3.
%rational singularities by the Grauert-Riemenschneider vanishing theorem
Assume that $X$ is a projective irreducible holomorphic symplectic manifold.
Let $Def(X)$ and $Def(Y)$ be the Kuranishi spaces of $X$ and $Y$.
Denote by $\psi:\X\rightarrow Def(X)$ 
the semi-universal deformation of $X$, by $0\in Def(X)$ the point
with fiber $X$, and let $X_t$ be the fiber over $t\in Def(X)$. 
Let $\bar{\psi}:\Y\rightarrow Def(Y)$ 
be the semi-universal deformation of $Y$, $\bar{0}\in Def(Y)$
its special point with fiber $Y$, and 
$Y_t$ the fiber over $t\in Def(Y)$.
The following is a fundamental theorem of Namikawa:

\begin{thm}
\label{thm-namikawa}
(\cite{namikawa-deformations}, Theorem 2.2)
The Kuranishi spaces
$Def(X)$ and $Def(Y)$ are both smooth of the same dimension. 
They can be replaced by open neighborhoods of $0\in Def(X)$
and $\bar{0}\in Def(Y)$, and denoted again by $Def(X)$ and $Def(Y)$, 
in such a way that 
there exists a natural proper surjective map
$f:Def(X)\rightarrow Def(Y)$ with finite fibers.
Moreover, for a generic point $t\in Def(X)$, 
$Y_{f(t)}$ is isomorphic to $X_t$. 
\end{thm}

$Def(X)$ is thus a branched covering of $Def(Y)$. 
We first observe that the covering is Galois,
in analogy to the case of the resolution of 
a simple singularity of a $K3$ surface \cite{brieskorn,burns-wahl}.

\begin{lem}
\label{lemma-galois-cover}
\begin{enumerate}
\item
\label{item-galois}
The neighborhoods $Def(X)$ of $0$ and $Def(Y)$ of $\bar{0}$ in Theorem 
\ref{thm-namikawa} may be chosen, so that
the map $f:Def(X)\rightarrow Def(Y)$ is a branched Galois covering.
\item
\label{item-Hodge-isemetry}
The Galois group acts on $H^*(X,\Integers)$
via monodromy operators, which preserve the Hodge structure.
The action on $H^{1,1}(X,\RealNumbers)$ is faithful 
and the action on $H^{2,0}(X)$ is trivial.
\end{enumerate}
\end{lem}

The elementary lemma is proven in section \ref{sec-proof}.
Significant generalizations of the lemma
are considered in section \ref{sec-generalizations}.
We proceed to calculate the Galois group.
Let $\Sigma\subset Y$ be the singular locus. 
The {\em dissident locus} $\Sigma_0\subset \Sigma$ is the locus along which the
singularities of $Y$ fail to be of $ADE$ type. 
$\Sigma_0$ is a closed subvariety of $\Sigma$.

\begin{prop} 
\label{prop-dissident-locus}
(\cite{namikawa-deformations}, Propositions 1.6)
$Y$ has only canonical singularities. The dissident locus 
$\Sigma_0$ has codimension at least $4$ in $Y$. 
The complement $\Sigma\setminus\Sigma_0$ is either empty, 
or the disjoint union of codimension $2$ smooth and symplectic subvarieties of 
$Y\setminus \Sigma_0$. 
\end{prop}

Kaledin proved that 
the morphism $\pi$ is {\em semi-small}, i.e., the fiber product
$X\times_YX$ has pure dimension $2n$,
providing further information on the dissident locus \cite{kaledin-semi-small}.

Let $B\subset [\Sigma\setminus\Sigma_0]$ be a connected component,
$E$ the exceptional locus of $\pi$, and $E_B:=\pi^{-1}(B)$.
Then $E_B$ is connected, of pure codimension $1$ in $X$, and each fiber 
$\pi^{-1}(b)$, $b\in B$, is a tree of smooth rational curves, whose dual graph
is a Dynkin diagram of type $ADE$, by the above proposition. 
We refer to this Dynkin diagram as the
{\em Dynkin diagram of the fiber}.
Fix $b\in B$. The fundamental group $\pi_1(B,b)$ acts on the
set of irreducible components of the fiber $\pi^{-1}(b)$ via isotopy classes of
diffeomorphisms of the fiber, and in particular via automorphisms
of the dual graph. We refer to the quotient, of  the
Dynkin diagram of the fiber by the 
$\pi_1(B,b)$-action, as the {\em folded Dynkin diagram} and denote by $W_B$
the Weyl group of the folded root system (see section
\ref{sec-folding}). 
Note that the set of irreducible components of $E_B$ corresponds to 
a basis of simple roots for the folded root system.
Let $\B$ be the set of connected components of $\Sigma\setminus\Sigma_0$.
Let $G$ be the Galois group of the branched Galois cover
$f: Def(X)\rightarrow Def(Y)$, introduced in Lemma \ref{lemma-galois-cover}.

\begin{thm}
\label{thm-G-is-isomorphic-to-product-of-weyl-groups}
$G$ is isomorphic to $\prod_{B\in\B}W_B$.
\end{thm}

\preprint{
The Theorem is the combined results of Lemmas 
\ref{lem-Weyl-group-action} and \ref{lem-Weyl-group-action-A-even}.}
\journal{The Theorem follows from Lemma
\ref{lem-Weyl-group-action}, excluding certain exceptional cases. 
The proof of the exceptional cases is outlined in section
\ref{sec-brief-description-of-folded-A-2k-case}. A detailed proof 
can be found in the preprint version
of this paper \cite{preprint-version}, Lemma 4.24.}
The $K3$ case of the Theorem was proven by Burns and Wahl 
\cite{burns-wahl}, Theorem 2.6 and Remark 2.7.
See also \cite{szendroi-enhanced}.
Y. Namikawa recently extended Theorem 
\ref{thm-G-is-isomorphic-to-product-of-weyl-groups}
to the case of Poisson deformation spaces of affine symplectic varieties
\cite{namikawa-Poisson-deformations-II}.

The above group $G$,
%Theorem \ref{thm-G-is-isomorphic-to-product-of-weyl-groups}
%calculates the local monodromy group $G$ of $X$, 
associated to the {\em birational} contraction $\pi:X\rightarrow Y$,
acts on $H^*(X,\Integers)$ via local monodromy operators, 
by Lemma \ref{lemma-galois-cover}. 
One can define abstractly the {\em local monodromy group} of $X$,  
which is larger in general (\cite{markman-monodromy-I}, Definition 2.4).
It includes, for example, monodromy operators
induced by automorphisms and birational automorphisms.
Affine Weyl groups can be constructed as subgroups of 
the local monodromy group, by taking the group generated
by Weyl groups $G_i$, associated to two or more 
birational contractions $\pi_i:X\rightarrow Y_i$ as 
in Theorem \ref{thm-G-is-isomorphic-to-product-of-weyl-groups}.
Consider, for example, a $K3$ surface $S$, admitting an
affine Dynkin diagram of rational curves as a fiber in an
elliptic fibration. Choose two Dynkin sub-diagrams of finite type, which union
is the whole affine diagram. 
Then the elliptic fibration factors through the two
different birational contractions, and the two Galois groups
will generate an affine Weyl group.
Affine Weyl group monodromy actions occur for quiver varieties
\cite{nakajima-reflections}.
A conjectural characterization of the local monodromy group
is discussed in \cite{markman-monodromy-I}, section 2.

%In the $K3$ surface case one can get an affine Weyl group,
%as a subgroup of the local monodromy group, by considering 
%affine Dynkin diagrams of rational curves appearing in
%fibers of elliptic fibrations. It seems natural to expect
%a similar picture to emerge in the case
%of higher dimensional completely integrable symplectic varieties.

The paper is organized as follows. 
Lemma \ref{lemma-galois-cover}, stating that the modular cover is Galois, 
is proven in section \ref{sec-proof}.
In section \ref{sec-folding} we recall the procedure of folding 
Dynkin diagrams of $ADE$ type by diagram automorphisms.
Section \ref{sec-Galois-groups} is dedicated to the proof of Theorem 
\ref{thm-G-is-isomorphic-to-product-of-weyl-groups}, i.e., the calculation
of the modular Galois group $G$ as a product of Weyl groups. 
The key step is the construction of a reflection in $G$, for every 
codimension one irreducible component of the exceptional 
locus of the contraction $\pi:X\rightarrow Y$ 
\preprint{
(Lemmas \ref{lem-galois-reflection-via-flop} and 
\ref{lem-galois-reflection-via-flop-A-2r-case}).}
\journal{(Lemma \ref{lem-galois-reflection-via-flop}).}
In section \ref{sec-examples} we apply Theorem 
\ref{thm-G-is-isomorphic-to-product-of-weyl-groups} 
to calculate modular Galois groups in specific examples. 
We consider, in particular, O'Grady's $10$-dimensional example
of an irreducible holomorphic symplectic manifold $X$
with $b_2(X)=24$. We apply Theorem 
\ref{thm-G-is-isomorphic-to-product-of-weyl-groups} to
explain the occurrence of the root lattice of $G_2$
as a direct summand of the lattice $H^2(X,\Integers)$ 
(Lemma \ref{lemma-Ogrady-Galois-group-is-W-G-2}).
The above is part of an earlier result of Rapagnetta \cite{rapagnetta}.

Section \ref{sec-generalizations} is devoted to 
generalizations of Lemma \ref{lemma-galois-cover}.
We consider birational contractions
$\pi:X\rightarrow Y$, where $Y$ is
affine symplectic, 
or the symplectic variety $Y$ does not admit a crepant resolution, 
or $Y$ is a Calabi-Yau threefold with a curve of $ADE$-singularities.
Again we prove that the associated modular morphism 
$f:Def(X)\rightarrow Def(Y)$ is a Galois branched cover of it image
(it need not be surjective in the non-symplectic cases).
The generalization in section \ref{sec-affine-examples}
is sufficient to explain the analogy between the results of 
the current paper and results of Nakajima on quiver varieties
\cite{nakajima-reflections}. Arguably, the most significant
generalization is Lemma \ref{lemma-generalized-galois-cover}.
The latter is used in Lemma \ref{lemma-CY-threefolds}
to explain the apparent analogy between the current 
paper and the string theory phenomena of enhanced gauge symmetries 
associated to Calabi-Yau threefolds with a curve of $ADE$-singularities
\cite{aspinwall,katz-morrison-plesser}.
These symmetries were explained mathematically by 
Szendr\Hun{o}i \cite{szendroi-enhanced,szendroi-artin-group}
and related, in a limiting set-up, to Hitchin systems of ADE-type
in \cite{donagi-pantev-A1,donagi-pantev-ADE}.

{\em Acknowledgments:} The author benefited from conversations
and correspondences with 
F. Catanese,
B. Hassett,
M. Lehn,
H. Nakajima, 
Y. Namikawa,
K. O'Grady, 
T. Pantev, 
A. Rapagnetta,
C. Sorger, and 
K. Yoshioka. 
I am deeply grateful to them all.
This paper depends heavily on the results of the paper 
\cite{namikawa-deformations}. I would like to thank 
Y. Namikawa for writing this fundamental paper, for reading
an early draft of the proof of Lemma
\ref{lemma-galois-cover}, for his encouragement, and 
for helpful suggestions. 
The author considered earlier a local monodromy reflection, associated
to the Hilbert-Chow contraction $S^{[n]}\rightarrow S^{(n)}$,
from the Hilbert scheme of length $n$ subschemes
on a $K3$ surface to the symmetric product 
\cite{markman-monodromy-I}, Lemma 2.7 and
\cite{markman-constraints}, Example 5.2. 
The relevance of Namikawa's work in this case was
pointed out to me by K. Yoshioka. I thank him for this, and for numerous 
instructive conversations.

%**********************************************************
%
%**********************************************************
\section{Modular branched Galois covers via local Torelli
%Proof of Lemma \ref{lemma-galois-cover}
}
\label{sec-proof}
We prove Lemma \ref{lemma-galois-cover} in this section.

Part (\ref{item-galois}):
We may assume that the neighborhoods $Def(X)$ and 
$Def(Y)$ are sufficiently small, 
so that the fiber of $f$  over $\bar{0}$ consists of
the single point $0$ corresponding to $X$.
The morphism $\pi:X\rightarrow Y$ deforms as a morphism $\nu$ of the 
semi-universal families, which fits in a commutative diagram 
\begin{equation}
\label{diagram-f-general}
\begin{array}{ccc}
\X & \RightArrowOf{\nu} & \Y
\\
\psi \ \downarrow \hspace{1ex} & & \bar{\psi} \ \downarrow \ \hspace{1ex}
\\
Def(X) & \RightArrowOf{f} & Def(Y),
\end{array}
\end{equation}
by \cite{kollar-mori}, Proposition 11.4.
Let $U\subset Def(Y)$ be the largest open subset satisfying the following 
conditions:
The semi-universal deformation $\Y$
restricts to a smooth family over $U$ and $f$ is unramified over $U$. 
Diagram (\ref{diagram-f-general})
restricts over the subset $U$
of $Def(Y)$ to a cartesian diagram, by Theorem
\ref{thm-namikawa}.
Set $V:=f^{-1}(U)$. 
%and let $h:\widetilde{V}\rightarrow V$ be 
%the Galois closure of $f:V\rightarrow U$, so that 
%the composition $\tilde{f}:=f\circ h:\widetilde{V}\rightarrow U$
%is a Galois covering. 
%If we choose a point $u\in U$,
%we get a homomorphism $\rho$ from $\pi_1(U,u)$ to the permutation group of
%the points in the fiber $f^{-1}(u)$. The kernel of $\rho$ is
%the normal subgroup of $\pi_1(U,u)$, which determines 
%the Galois closure $\widetilde{V}$.
The group $H^2(X,\Integers)$ is endowed with the monodromy invariant 
Beauville-Bogomolov symmetric bilinear pairing \cite{beauville}.
Set $\Lambda:=H^2(X,\Integers)$ and let 
\[
\Omega_\Lambda \ := \ \{
\ell\in \PP{H}^2(X,\ComplexNumbers) \ : \ (\ell,\ell)=0 \ \ \ \mbox{and} \ \ \ 
(\ell,\bar{\ell}) > 0\} 
\]
be the period domain. 
The local system $R^2_{\psi_*}(\Integers)$ is trivial.
Choose the trivialization $\varphi:R^2_{\psi_*}(\Integers)\rightarrow \Lambda$,
inducing the identity on $H^2(X,\Integers)$.
Let $p:Def(X)\rightarrow \Omega_\Lambda$ be the period map, given by 
$p(t)=\varphi(H^{2,0}(X_t))$. Recall that $p$ is a holomorphic embedding 
of $Def(X)$ as an open subset of $\Omega_\Lambda$, by
the Local Torelli Theorem \cite{beauville}.
%Set $\tilde{p}:=p\circ h$. 
We get the following diagram:
\begin{equation}
\label{main-diagram}
\begin{array}{ccccc}
& & \X & \LongRightArrowOf{\nu} & \Y 
\\
& & \hspace{1ex} \ \downarrow \ \psi & & \hspace{1ex} \ \downarrow \ \bar{\psi}
\\
\Omega_\Lambda & \LongLeftArrowOf{p} & Def(X) & \LongRightArrowOf{f} &
Def(Y)
\\
%\tilde{p} \ \uparrow \ \hspace{1ex} 
&  & \cup & & \cup 
\\
%\widetilde{V} 
& 
%\LongRightArrowOf{h} 
& V & \longrightarrow & U.
\end{array}
\end{equation}

%We need to prove that $h:\widetilde{V}\rightarrow V$ is the identity.
%Set $G:=Gal(\widetilde{V}/U)$ and let $\Aut(Def(X))$ be the group of 
%biholomorphic automorphisms of $Def(X)$. It suffices to construct 
%a homomorphism $\alpha:G\rightarrow \Aut(Def(X))$, with respect to which $h$
%is $G$-equivariant. Indeed, such a homomorphism $\alpha$ must
%be injective, by the minimality of the Galois closure, since the Galois covering
%$\widetilde{V}/\ker(\alpha)\rightarrow U$ factors through $V$.

Choose a point $u\in U$. We get the monodromy homomorphism
${\rm mon}:\pi_1(U,u)\rightarrow O(\Lambda):=O\left[H^2(X,\Integers)\right]$.
Set $K:=\ker({\rm mon})$ and let $f_K:\widetilde{U}_K\rightarrow U$
be the Galois\footnote{I thank F. Catanese for the suggestion to consider
this Galois cover. It simplifies the original proof,
which considered instead the Galois closure of $f:V\rightarrow U$.} 
cover of $U$ associated to $K$.

We have a natural isomorphism 
\begin{equation}
\label{eq-pull-back-by-nu-is-an-isomorphism}
\nu^* \ : \ f^*\left[R^2_{\bar{\psi}_*}(\Integers\restricted{)}{U}\right]
\ \ \ \LongIsomRightArrow \ \ \ 
R^2_{\psi_*}(\Integers\restricted{)}{V}.
\end{equation}
Hence, the local system $f^*\left[R^2_{\bar{\psi}_*}(\Integers\restricted{)}{U}\right]$
is trivial. It follows that the covering $f:V\rightarrow U$ factors as $f=h\circ f_K$ 
via a covering $h:V\rightarrow \widetilde{U}_K$.
The restriction of the period map $p$ to $V$ clearly factors through $h$. 
Now $p$ is injective. Hence, so is $h$.
We conclude that $h$ is an isomorphism.

Set $G:=\pi_1(U,u)/K$. 
It remains to extend the action of $G$ on $V$ to an action on $Def(X)$.
Set $\phi:=\varphi\circ\nu^*:
f^*\left[R^2_{\bar{\psi}_*}(\Integers\restricted{)}{U}\right]\rightarrow 
\Lambda$.
Let $\gamma:G\rightarrow O(\Lambda)$ be the
homomorphism induced by the monodromy representation. 
Given a deck transformation $g\in G$, we get the commutative diagram
of trivializations of local systems over $V$:
\begin{equation}
\label{eq-gamma-g}
\begin{array}{ccc}
f^*\left[R^2_{\bar{\psi}_*}(\Integers\restricted{)}{U}\right] & 
\LongRightArrowOf{\phi} & \Lambda
\\
= \ \downarrow \ \hspace{1Ex} & & \hspace{3ex} \ \downarrow \ \gamma_g
\\
f^*\left[R^2_{\bar{\psi}_*}(\Integers\restricted{)}{U}\right] & 
\LongRightArrowOf{(g^{-1})^*\phi} & \Lambda.
\end{array}
\end{equation}
%for some isometry $\gamma_g\in O\Lambda$. Clearly,
%$\gamma_{g^{-1}}=\gamma_g^{-1}$. 
%Given $t\in \widetilde{V}$, and regarding $\gamma_g$ as an automorphism
%of the trivial local system, we have\\
%${\displaystyle
%(\gamma_{g_1g_2})_t=
%\left((g_2^{-1}g_1^{-1})^*h^*\phi\right)_t\circ (h^*\phi^{-1})_t=
%\phi_{h(g_2^{-1}(g_1^{-1}(t)))}\circ\phi_{h(t)}^{-1}=
%}$
%\\
%${\displaystyle
%\phi_{h(g_2^{-1}(g_1^{-1}(t)))}\circ
%\phi^{-1}_{h(g_1^{-1}(t))}\circ \phi_{h(g_1^{-1}(t))}
%\circ\phi_{h(t)}^{-1}=
%(\gamma_{g_2})_{g_1^{-1}(t)}\circ (\gamma_{g_1})_t=
%(\gamma_{g_2})_t\circ (\gamma_{g_1})_t.
%}$\\ 
%We conclude that the map $g\mapsto \gamma_g$ defines a 
%group homomorphism $\gamma:G\rightarrow O\Lambda$. 
%Regarding $G$ as the quotient of $\pi_1(U,u)$, as above, 
%the homomorphism $\gamma$ is the monodromy
%representation of $R^2_{\bar{\psi_*}}(\Integers\restricted{)}{U}$. 

The group $O(\Lambda)$ acts on 
the period domain and we denote the automorphism of 
$\Omega_\Lambda$ corresponding to $\gamma_g$ by $\gamma_g$ as well.
We have\\
${\displaystyle
\gamma_g(p(t))=
\gamma_g\left(\varphi_{t}(H^{2,0}(X_{t}))\right)=
\gamma_g\left(\phi_{t}(H^{2,0}(Y_{f(t)}))\right)
\stackrel{(\ref{eq-gamma-g})}{=}}$
\\
${\displaystyle \phi_{g^{-1}(t)}(H^{2,0}(Y_{f(t)}))=
%\varphi_{h(g^{-1}(t))}\left(\nu^*_{h(g^{-1}(t))}(H^{2,0}(Y_{\tilde{f}(t)}))
%\right)=}$
%\\
%${\displaystyle
\varphi_{g^{-1}(t)}(H^{2,0}(X_{g^{-1}(t)}))=p(g^{-1}(t)).
}
$\\
We conclude the equality
\begin{equation}
\label{eq}
p\circ g^{-1} \ \ \  = \ \ \ \gamma_g\circ  p.
\end{equation}

\begin{claim}
\begin{enumerate}
\item
\label{claim-item-image-of-p-is--gamma-invariant}
We can choose the above open neighborhood $Def(Y)$ of $\bar{0}$ 
in the Kuranishi space of $Y$, and similarly $Def(X)$ in that of $X$, 
to satisfy 
\[
\gamma_g(p(Def(X))=p(Def(X)).
\]
\item
\label{claim-item-period-of-0-is-gamma-invariant}
$\gamma_g(p(0))=p(0)$.
\end{enumerate}
\end{claim}

\begin{proof}
(\ref{claim-item-image-of-p-is--gamma-invariant}) 
The equality (\ref{eq}) yields
$\gamma_g(p(V))=p\left(g^{-1}(V)\right)=p(V)$.
The period map $p$ is open
%\footnote{The argument requires only that 
%$p(Def(X))$ is an open subset of a closed analytic subvariety of the 
%period domain.} 
and biholomorphic onto its image, so the image 
$p(Def(X))$ is the interior
of its closure in $\Omega_\Lambda$, possibly after replacing $Def(X)$ and 
$Def(Y)$ by smaller open neighborhoods. 
The closure of $p(Def(X))$
is equal to the closure of $p(V)$ and is thus $\gamma_g$-invariant.

(\ref{claim-item-period-of-0-is-gamma-invariant}) 
Define the homomorphism $\alpha : G\rightarrow \Aut(Def(X))$ by 
\[
\alpha_g \ \ \  := \ \ \ p^{-1}\circ \gamma_g\circ p.
\] 
We have
\[
f\circ \alpha_g=
f\circ p^{-1}\circ \gamma_g \circ p
\stackrel{(\ref{eq})}{=}
f\circ p^{-1}\circ p\circ g^{-1}=
f\circ g^{-1}= f.
\]
%Hence, $f\circ \alpha_g=f$. 
Thus $\alpha_g(0)$ belongs to the fiber
of $f$ over $f(0)$. This fiber consists of a single point,
so $\alpha_g(0)=0$.
\end{proof}

%We have
%$
%p\circ h\circ g^{-1}
%\stackrel{(\ref{eq})}{=}\gamma_g\circ \tilde{p}=p\circ \alpha_g\circ h.
%$
%The equality $h\circ g^{-1}=\alpha_g\circ h$ follows, since $p$ is injective.
%Hence, $h$ is $G$-equivariant.

Proof of Lemma \ref{lemma-galois-cover} part (\ref{item-Hodge-isemetry}): 
The monodromy action of $G$
on $H^2(X,\Integers)$ is given by the homomorphism $\gamma$. 
Let $\tilde{\gamma}:G\rightarrow GL[H^*(X,\Integers)]$
be the  monodromy representation and consider the analogue of
Diagram (\ref{eq-gamma-g}), where $\tilde{f}=f$, 
$\tilde{\varphi}:R^*_{\psi_*}(\Integers)\nopagebreak \rightarrow\nopagebreak 
H^*(X,\Integers)$
is the trivialization, inducing the identity on the fiber $H^*(X,\Integers)$
over $0$, 
$\tilde{\phi}:=\tilde{\varphi}\circ
\nu^*:f^*\left[R^*_{\bar{\psi}_*}(\Integers\restricted{)}{U}\right]
\rightarrow H^*(X,\Integers)$, and
$\tilde{\gamma}_g\circ\tilde{\phi}=(g^{-1})^*\tilde{\phi}$.
Given $t\in V$, we get that
\[
\tilde{\varphi}^{-1}_{g^{-1}(t)}\tilde{\gamma}_g\tilde{\varphi}_t:
H^*(X_t,\Integers)
\rightarrow
H^*(X_{g^{-1}(t)},\Integers)
\]
is induced by pullback via  
$X_{g^{-1}(t)}\IsomRightArrowOf{\nu_{g^{-1}(t)}} 
Y_{f(t)}\IsomRightArrowOf{\nu^{-1}_t}X_t$
and thus preserves the Hodge structure.
Letting $t$ approach $0\in Def(X)$, we get that 
$\tilde{\varphi}^{-1}_{g^{-1}(0)}\tilde{\gamma}_g
\tilde{\varphi}_0:H^*(X_0,\Integers)\rightarrow H^*(X_{g^{-1}(0)},\Integers)$ 
preserves the Hodge structure as well.
The statement follows, since $g(0)=\alpha_g(0)=0$,
by the above claim, and $\tilde{\varphi}_0=id$. 
The subspace $H^{2,0}(X)$ is a trivial representation of $G$, since 
$H^{2,0}(X)$ is contained in the trivial\footnote{
The triviality of the monodromy representation 
$\pi^*H^2(Y,\ComplexNumbers)$ will be proven in detail 
in Proposition \ref{prop-G-embedds-in-G-L}.} 
representation $\pi^*H^2(Y,\ComplexNumbers)$,
by \cite{kaledin-semi-small}, Lemma 2.7.
\EndProof

%****************************************************************
%
%****************************************************************
\section{Folding Dynkin diagrams}
\label{sec-folding}

Consider a simply laced root system $\widetilde{\Phi}$ 
with root lattice $\Lambda_r$
and a basis of simple roots $\F:=\{f_1, \dots, f_r\}$. 
We regard the simple roots also as the nodes of the
Dynkin graph. The bilinear pairing on $\Lambda_r$ is determined by the
Dynkin graph as follows: 
\[
(f_i,f_j)=\left\{
\begin{array}{ccl}
2 & \mbox{if} & i=j,
\\
-1 & \mbox{if} & i\neq j \ \mbox{and the nodes} \ f_i \ \mbox{and} \ f_j \
\mbox{are adjacent},
\\
0 & \mbox{if} & i\neq j \ \mbox{and the nodes} \ f_i \ \mbox{and} \ f_j \
\mbox{are not adjacent}.
\end{array}\right.
\]
Let $\Gamma$ be a subgroup of the automorphism group of the Dynkin graph. 
If non-trivial, then $\Gamma=\Integers/2\Integers$, for types $A_n$, $n\geq 2$,
$D_n$, $n\geq 5$, and $E_6$, and $\Gamma$ is a subgroup of the
symmetric group $Sym_3$, for type $D_4$
\cite{humphreys}, section 12.2.
Given an orbit $j:=\Gamma\cdot f_{\tilde{j}}\in \F/\Gamma$, 
let $e_j$ be the orthogonal projection of $f_{\tilde{j}}$ from
$\Lambda_r$ to the $\Gamma$-invariant subspace 
$[\Lambda_r\otimes_\Integers\RationalNumbers]^\Gamma$.
The projection clearly depends only on the orbit $\Gamma\cdot f_{\tilde{j}}$.

We recall the construction of the {\em folded root system}
(also known as the root system of the twisted group associated to
an outer automorphism, see \cite{carter}, Ch. 13). 
Its root lattice is the lattice 
$\Lambda_{\bar{r}}:=\mbox{span}_\Integers\{e_1, \dots, e_{\bar{r}}\}$, 
where $\bar{r}$ is the cardinality of $\F/\Gamma$.
Set $e_i^\vee:=\frac{2(e_i,\bullet)}{(e_i,e_i)}$, $i\in \F/\Gamma$.
Let $f_{\tilde{i}}$ and $f_{\tilde{j}}$ be two simple roots and 
set $i:=\Gamma\cdot f_{\tilde{i}}$ and $j:=\Gamma\cdot f_{\tilde{j}}$.

Examination of the Dynkin graphs of ADE type shows that there are
only two types of $\Gamma$ orbits:
\begin{enumerate}
\item
\label{case-of-non-adjacent-rrots}
The orbit $\Gamma\cdot f_{\tilde{j}}$ consists of pairwise non-adjacent
roots.
\item
\label{case-of-two-adjacent-roots}
The original root system is of type $A_{2n}$, $\Gamma=\Integers/2\Integers$,
and the orbit $\Gamma\cdot f_{\tilde{j}}$ consists of the two
middle roots of the Dynkin graph.
\end{enumerate}

\preprint{
The following Lemma will be used in the proofs of Lemmas 
\ref{lem-Weyl-group-action} and \ref{lem-Weyl-group-action-A-even}.}
\journal{The following Lemma will be used in the proof of Lemma
\ref{lem-Weyl-group-action}.}

\begin{lem}
\label{lemma-folding}
The following equation holds:
\\
${\displaystyle 
e_j^\vee(e_i) \ = \ 
\left\{\begin{array}{ccl}
2 & \mbox{if} & i=j,
\\
\sum_{f_{\tilde{k}}\in\Gamma\cdot f_{\tilde{j}}}
\left(f_{\tilde{k}},f_{\tilde{i}}\right)
& \mbox{if} & i\neq j \ \mbox{and} \ 
\mbox{type}(\Gamma\cdot f_{\tilde{j}})=1,
\\
2\sum_{f_{\tilde{k}}\in\Gamma\cdot f_{\tilde{j}}}
\left(f_{\tilde{k}},f_{\tilde{i}}\right)
& \mbox{if} & i\neq j \ \mbox{and} \ 
\mbox{type}(\Gamma\cdot f_{\tilde{j}})= 2.
\end{array}
\right. 
}$
\\
Consequently, $e_j^\vee(e_i)$ is an integer and the reflection 
$\rho_j(x)=x-e_j^\vee(x)e_j$ by $e_j$ maps $\Lambda_{\bar{r}}$
%$\mbox{span}_\Integers\{e_1, \dots, e_{\bar{r}}\}$ 
onto itself. 
\end{lem}

\begin{proof}
${\displaystyle 
e_j^\vee(e_i)=\frac{2(e_j,e_i)}{(e_j,e_j)}=
\frac{2(e_j,f_{\tilde{i}})}{
(e_j,f_{\tilde{j}})}=
\frac{2\left(\sum_{\gamma\in\Gamma}\gamma(f_{\tilde{j}}),f_{\tilde{i}}\right)}
{\left(\sum_{\gamma\in\Gamma}\gamma(f_{\tilde{j}}),f_{\tilde{j}}\right)}.
}$
\\
If the orbit $\Gamma\cdot f_{\tilde{j}}$ is of type
\ref{case-of-non-adjacent-rrots},
the denominator 
$\sum_{\gamma\in\Gamma}\left(\gamma(f_{\tilde{j}}),f_{\tilde{j}}\right)$
is equal to twice the order of the subgroup of $\Gamma$ stabilizing 
$f_{\tilde{j}}$. Thus, $e_j^\vee(e_i)=
\sum_{f_{\tilde{k}}\in\Gamma\cdot f_{\tilde{j}}}(f_{\tilde{k}},f_{\tilde{j}})$.
If the orbit $\Gamma\cdot f_{\tilde{j}}$ is of type
\ref{case-of-two-adjacent-roots}, let $\gamma$ be the non-trivial 
element of $\Gamma$. Then the denominator is equal to 
$(f_{\tilde{j}},f_{\tilde{j}})+(\gamma(f_{\tilde{j}}),f_{\tilde{j}})=2-1=1$.
Furthermore, $\Gamma=\Integers/2\Integers$ acts freely on the
orbit of $f_{\tilde{j}}$.
\end{proof}

The Weyl group $W$,
of the folded root system, is the subgroup of the isometry group 
of $\Lambda_{\bar{r}}$, generated by the reflections with respect to
$e_i$, $1\leq i\leq \bar{r}$. $W$ is finite, since the bilinear pairing
is positive definite.
The set of roots $\Phi$ is the union 
of $W$-orbits of simple roots $\cup_{i=1}^{\bar{r}}W\cdot e_i$. 
We conclude that $\Phi$ is a, possibly non-reduced, root system
(i.e., it satisfies axioms R1, R3, R4 in \cite{humphreys}, section 9.2,
but we have not\footnote{See however 
\cite{humphreys} Section 12.2 Exercise 3, stating that
the only irreducible non-reduced root systems are of type $BC_n$.} 
verified axiom R2, that the only multiples of
$\alpha\in \Phi$, which are also in $\Phi$, are $\pm\alpha$). 
 
Following is the list of folded root systems, which can be found in
\cite{carter}, section 13.3.
If $\widetilde{\Phi}$ is of type $A_r$, $r\geq 2$, $D_r$, $r\geq 5$, 
or $E_6$, then the automorphism group $\Gamma$ of the Dynkin diagram is
$\Integers/2\Integers$. Setting $\Phi$ to be the folded root system, the
types of the pairs $(\widetilde{\Phi},\Phi)$ are: 
$(A_{2n},B_n)$, $(A_{2n-1},C_n)$, $(D_r,B_{r-1})$, and $(E_6,F_4)$. 
When  $\widetilde{\Phi}$ is of type $D_4$, the automorphism group
of the Dynkin diagram is the symmetric group $Sym_3$. 
The folding by a subgroup $\Gamma$ of order $2$ results
in $\Phi$ of type $C_3$. Let us work out the remaining case explicitly.

\begin{example}
\label{example-folding-D-4}
Consider the root system of type $D_4$ with simple roots 
$f_1, f_2, f_3, f_4$, whose Dynkin graph has three edges from $f_2$
to each of the other simple roots. 
Let $\Gamma$ be either the full automorphism group $Sym_3$, 
permuting the roots $\{f_1, f_3, f_4\}$, or its cyclic subgroup of order $3$.
Set $e_1:=\frac{f_1+f_3+f_4}{3}$ and $e_2:=f_2$. 
Then $e^\vee_1=3e_1$ and $e^\vee_2=e_2$. 
The intersection matrix of the root lattice $\mbox{span}_\Integers\{e_1,e_2\}$
is 
$\left((e_i,e_j)\right)=\left(\begin{array}{cc}
2/3 & -1
\\
-1 & 2
\end{array}\right)$
and its Cartan matrix is
$\left((e_i,e^\vee_j)\right)=\left(\begin{array}{cc}
2 & -1
\\
-3 & 2
\end{array}\right)$.
Hence, the folded Dynkin diagram is that of $G_2$.
This example is revisited in section \ref{sec-ogrady-10-dimensional-example}.
\end{example}

%****************************************************************
%
%****************************************************************
\section{Galois groups}
\label{sec-Galois-groups}
We prove Theorem 
\ref{thm-G-is-isomorphic-to-product-of-weyl-groups} in this section.
Let $\pi:X\rightarrow Y$ be a symplectic projective resolution of a normal 
projective variety $Y$ of complex dimension $2n$. 
Assume that $X$ is an irreducible holomorphic symplectic manifold. 
We keep the notation of Proposition \ref{prop-dissident-locus}
and Theorem \ref{thm-G-is-isomorphic-to-product-of-weyl-groups}.

%***********
\hide{
%***********
The above question is related to the Torelli question. 
Let $g$ be a Hodge isometry of
$H^2(X,\Integers)$, which is induced by monodromy. 
Then $g$ acts on $H^{1,1}(X)$, and hence on $H^1(X,T_X)$. 
Let $g^*\X\rightarrow g^{-1}(Def(X))$ be the pull-back 
of the semi-universal deformation of $X$. 
Note that $0$ belongs to $Def(X)\cap g^{-1}(Def(X))$.
Given a subset $A$ of the latter, let $U_A$ be
the complement of $A$ in $Def(X)\cap g^{-1}(Def(X))$.
A version of the Torelli question asks: 
\begin{question}
Is there a closed analytic proper subset $A\subset Def(X)\cap g^{-1}(Def(X))$,
such that the map $g:U_A\rightarrow g(U_A)$ lifts to an isomorphism
$\tilde{g}:g^*\left(\X_{U_A}\right)\rightarrow \X_{g(U_A)}$ 
of the restricted  families?
\end{question}
}
%**********
%End hide
%**********

%****************************************************************
%
%****************************************************************
\subsection{Two Leray filtrations of $H^2(X,\Integers)$
%The saturation and co-rank of $\pi^*H^2(Y,\Integers)$ in $H^2(X,\Integers)$.
}
Set $U:=Y\setminus \Sigma_0$ and $\widetilde{U}:=X\setminus\pi^{-1}(\Sigma_0)$.
Let $\phi:\widetilde{U}\rightarrow U$ be the restriction of $\pi$ to 
$\widetilde{U}$. 
We get the commutative diagram of pull-back homomorphisms
\[
\begin{array}{ccc}
H^2(Y,\Integers) & \LongRightArrowOf{\iota_U^*} & H^2(U,\Integers)
\\
\pi^* \ \downarrow \ \hspace{2ex} & & \hspace{2ex} \ \downarrow \ \phi^*
\\
H^2(X,\Integers) & \LongRightArrowOf{\iota_{\widetilde{U}}^*}&
H^2(\widetilde{U},\Integers).

\end{array}
\]

\begin{lem}
\label{lemma-cohomology-of-X-versus-Y}
\hspace{1ex}
\begin{enumerate}
\item
\label{lemma-item-iota-U-tilde-is-an-isom}
The homomorphism 
$\iota_{\widetilde{U}}^*:
H^2(X,\Integers)\rightarrow H^2(\widetilde{U},\Integers)$ is bijective.
%\item
%\label{lemma-item-iota-U-is-injective}
%The homomorphism $\iota^*_U:H^2(Y,\Integers)\rightarrow H^2(U,\Integers)$ 
%is injective.
\item
\label{lemma-item-pullback-by-phi-is-injective}
The homomorphism $\phi^*:H^2(U,\Integers)\rightarrow 
H^2(\widetilde{U},\Integers)$ is injective.
\item
\label{lemma-item-saturated}
Denote the image of the composition
\[
H^2(U,\Integers)\LongRightArrowOf{\phi^*} 
H^2(\widetilde{U},\Integers)
\LongRightArrowOf{(\iota_{\widetilde{U}}^*)^{-1}} H^2(X,\Integers)
\]
by $H^2(U,\Integers)$ as well.
Then 
%image $\pi^*H^2(Y,\Integers)$ 
$H^2(U,\Integers)$ 
is saturated in $H^2(X,\Integers)$.
\item
\label{lemma-item-corank-of-saturated-sublattice}
Let $\E$ be the set of all irreducible components
of codimension $1$ of the exceptional locus of $\pi$.
The subset $\{[E] \ : \ E\in \E\}$ of $H^2(X,\Integers)$
is linearly independent. $\E$ maps bijectively onto 
a basis of $H^2(X,\RationalNumbers)/H^2(U,\RationalNumbers)$ via
the map $E\mapsto [E]+H^2(U,\RationalNumbers)$.
%The set of classes in $H^2(X,\Integers)$, of the irreducible components
%of codimension $1$ of the exceptional locus of $\pi$,
%projects to a basis of $H^2(X,\Integers)/H^2(U,\Integers)$.
%In particular, this set is linearly independent 
%and its cardinality is equal to the co-rank of 
%$H^2(U,\Integers)$ in $H^2(X,\Integers)$.
\end{enumerate}
\end{lem}

\begin{proof}
\ref{lemma-item-iota-U-tilde-is-an-isom})
The inverse image $\widetilde{\Sigma}_0:=\pi^{-1}(\Sigma_0)$ 
has complex codimension $\geq 2$ in $X$, 
since $\Sigma_0$ has codimension $\geq 4$
in $Y$, by Proposition \ref{prop-dissident-locus}, and
$\pi$ is semi-small \cite{kaledin-semi-small}. 
$\widetilde{U}$ is thus simply connected, since $X$ is assumed so.
Furthermore, 
the groups $H^i(\widetilde{\Sigma}_0,\Integers)$ vanish, 
for $i=4n-2$ and $4n-3$. The top horizontal homomorphism in
the following commutative diagram is thus an isomorphism.
\[
\begin{array}{ccc}
H^{4n-2}(X,\widetilde{\Sigma}_0,\Integers)&\longrightarrow &
H^{4n-2}(X,\Integers)
\\
L.D. \ \downarrow \ \hspace{2Ex} & & \hspace{2Ex} \ \downarrow \ P.D.
\\
H_2(\widetilde{U},\Integers)&\LongRightArrowOf{\iota_{\widetilde{U}_*}}&
H_2(X,\Integers)
\end{array}
\]
The right arrow is the Poincare Duality isomorphism and the left is the 
Lefschetz Duality isomorphism 
(see \cite{munkres}, Theorem 70.6, as well as the argument in 
\cite{voisin-book-vol2}, proving equation (1.9) in the proof of Theorem 1.23).
The bottom push-forward homomorphism $\iota_{\widetilde{U}_*}$
is thus an isomorphism as well.
$H^2(X,\Integers)$ is isomorphic to $H_2(X,\Integers)^*$
and $H^2(\widetilde{U},\Integers)$  is isomorphic to 
$H_2(\widetilde{U},\Integers)^*$, by the Universal Coefficients Theorem
and the fact that both spaces are simply connected.
The restriction homomorphism 
$
\iota_{\widetilde{U}}^*:
H^2(X,\Integers)\rightarrow 
H^2(\widetilde{U},\Integers)
$
is thus an isomorphism.

\ref{lemma-item-pullback-by-phi-is-injective})
%\ref{lemma-item-iota-U-is-injective})
Associated to $\pi$ we have the canonical filtration $L$ on $H^q(X,\Integers)$
and the Leray spectral sequence $E_r^{p,q}(X)$, converging to 
$H^{p+q}(X,\Integers)$, with $E_2^{p,q}(X)=H^p(Y,R^q_{\pi_*}\Integers)$,
and $E_\infty^{p,q}(X)=Gr_L^{p,q}H^{p+q}(X,\Integers)$ 
(\cite{voisin-book-vol2}, Theorem 4.11).
Similarly we have the Leray spectral sequence
$E_r^{p,q}(\widetilde{U})$ associated to $\phi$ and converging to
$H^{p+q}(\widetilde{U},\Integers)$.
%The isomorphism
%$\iota_{\widetilde{U}}^*:H^2(X,\Integers)\rightarrow 
%H^2(\widetilde{U},\Integers)$ is compatible with the Leray filtrations, so
%it induces an injective homomorphism of the first graded summand 
%$E_\infty^{2,0}(X)\hookrightarrow E_\infty^{2,0}(\widetilde{U})$.
We have the equalities $E_2^{2,0}(X)=H^2(Y,\Integers)$ and
$E_2^{2,0}(\widetilde{U})=H^2(U,\Integers)$,
since $\pi$ has connected fibers, by Zariski's Main Theorem.

The sheaf $R^1_{\phi_*}\Integers$ vanishes, since it is supported
on $\Sigma\setminus\Sigma_0$ and over each connected component
$B$ of the latter, 
$\pi^{-1}(B)\rightarrow B$ is a topological fibration with 
simply connected fibers, by Proposition
\ref{prop-dissident-locus}.
Consequently, 
$E_2^{p,1}(\widetilde{U})$ vanishes, for all $p$, and 
the differentials
$d_2^{p,1}:H^p(U,R^1_{\phi_*}\Integers)\rightarrow H^{p+2}(U,\Integers)$
and
$d_2^{p,2}:H^p(U,R^2_{\phi_*}\Integers)\rightarrow 
H^{p+2}(U,R^1_{\phi_*}\Integers)$
vanish, for all $p$.
We get the equality $E_\infty^{p,q}(\widetilde{U})=E_2^{p,q}(\widetilde{U})$,
for $p+q=2$, and in particular 
$E_\infty^{2,0}(\widetilde{U})=H^2(U,\Integers)$.
The injectivity of $\phi^*$ follows.
%
%We have the equality:
%\[
%E_\infty^{2,0}(X) = 
%\mbox{coker}[d_2^{0,1}:
%H^0(Y,R^1_{\pi_*}\Integers)\rightarrow H^2(Y,\Integers)].
%\]
%Thus, the composition
%$
%H^2(Y,\Integers)\rightarrow E_\infty^{2,0}(X)
%\LongIsomRightArrowOf{\iota^*_{\widetilde{U}}} 
%E_\infty^{2,0}(\widetilde{U})
%\cong H^2(U,\Integers)
%$
%is surjective, as stated.
%
%The injectivity of $\phi^*$ 
%is equivalent to the equality 
%$E_\infty^{2,0}(\widetilde{U})=H^2(U,\Integers)$ established above.

\ref{lemma-item-saturated})
We need to show that $H^2(X,\Integers)/H^2(U,\Integers)$ is torsion free.
It suffices to show that $H^2(\widetilde{U},\Integers)/\phi^*H^2(U,\Integers)$
is torsion free, by part
\ref{lemma-item-iota-U-tilde-is-an-isom}. 
%and 
%\ref{lemma-item-iota-U-is-an-surjective}.
We have seen that $\phi^*H^2(U,\Integers)$ is isomorphic to
$E_\infty^{2,0}(\widetilde{U})$. 
$E_\infty^{1,1}(\widetilde{U})$ vanishes, by the vanishing of 
$E_2^{1,1}(\widetilde{U})$ observed above.  
Hence, it suffices to show that 
$E_{\infty}^{0,2}(\widetilde{U})$ is torsion free. 
Now $E_{\infty}^{0,2}(\widetilde{U})$ is 
isomorphic to $E_2^{0,2}(\widetilde{U}):=H^0(U,R^2_{\phi_*}\Integers)$.
The sheaf $R^2_{\phi_*}\Integers$ is supported 
as a local system over $\Sigma\setminus\Sigma_0$.
Its fiber, over a point $b$ in a 
connected component $B$ of $\Sigma\setminus\Sigma_0$,
is the free abelian group generated by the fundamental classes of the
irreducible components of the fiber $\pi^{-1}(b)$.
In particular, $H^0(U,R^2_{\phi_*}\Integers)$ is torsion-free.

\ref{lemma-item-corank-of-saturated-sublattice})
The fundamental group $\pi_1(B,b)$ permutes the above mentioned 
basis of the fiber of the sheaf $R^2_{\phi_*}\Integers$, 
over the point $b$. Hence, $H^0(U,R^2_{\phi_*}\Integers)$
has a basis consisting of the sum of elements
in each $\pi_1(B,b)$-orbit of \nolinebreak the original
basis elements. These orbits are in one-to-one correspondence
with the irreducible components of the
exceptional divisor of $\phi$. 
\end{proof}

\noindent
{\bf Caution:} We get the inclusions
$\pi^*H^2(Y,\Integers)\subset H^2(U,\Integers)\subset H^2(X,\Integers)$.
The first inclusion may be strict, i.e., the pullback homomorphism 
$\iota_{U}^*:H^2(Y,\Integers)\rightarrow 
H^2(U,\Integers)$ is {\em not} surjective in general.
%This surjectivity was erroneously claimed in an earlier version. 
We thank the referee for pointing out the following counter example. 
Let $\pi:X\rightarrow Y$ be a contraction of a
Lagrangian $\PP^n$, $n\geq 2$, to a point. 
Then $\pi^*H^2(Y,\Integers)\rightarrow H^2(X,\Integers)$ is not surjective,
since its image does not contain any ample line-bundle.
Now $U=\widetilde{U}$, so
$\iota_{U}^*:H^2(Y,\Integers)\rightarrow 
H^2(U,\Integers)$ must fail to be surjective.
The Leray filtration of $H^2(X,\Integers)$ is thus different from 
the image of the Leray filtration of $H^2(\widetilde{U},\Integers)$
via the isomorphism $\iota_{\widetilde{U}}^*$.

%****************************************************************
%
%****************************************************************
\subsection{An upper estimate of the Galois group}
Let $L$ be the kernel of the composition
\[
H^2(X,\Integers)\LongIsomRightArrowOf{P.D.} 
H_{4n-2}(X,\Integers)\LongRightArrowOf{\pi_*} H_{4n-2}(Y,\RationalNumbers), 
\]
where $2n:=\dim_\ComplexNumbers(X)$. $L$ is the saturation 
of the sublattice generated by the 
irreducible components of the exceptional divisors of the map
$\pi:X\rightarrow Y$. 

\begin{lem}
\label{lemma-L}
The Beauville-Bogomolov pairing restricts to a negative definite pairing on 
$L$. Furthermore,
$L^\perp$
%\pi^*H^2(Y,\Integers)$.
is equal to the subspace 
$H^2(U,\Integers)$ of $H^2(X,\Integers)$
defined in Lemma \ref{lemma-cohomology-of-X-versus-Y} part
\ref{lemma-item-saturated}.
\end{lem}

\begin{proof}
Choose a class $\sigma\in H^{2,0}(X)$, such that 
$\int_X(\sigma\bar{\sigma})^n=1$. 
There exists a positive real number $\lambda$, such that
\begin{eqnarray*}
2(\alpha,\beta) & =&  
\lambda n\int_X\alpha\beta(\sigma\bar{\sigma})^{n-1}+
\\
& & \lambda(1-n)\left\{
\int_X\alpha\sigma^{n-1}\bar{\sigma}^n\int_X\beta\sigma^n\bar{\sigma}^{n-1}+
\int_X\beta\sigma^{n-1}\bar{\sigma}^n\int_X\alpha\sigma^n\bar{\sigma}^{n-1}
\right\},
\end{eqnarray*}
by \cite{beauville}, Theorem 5.
%The Beauville-Fujiki relation states, that 
%there exists a positive real number $c$, such that
%$(\alpha,\alpha)^n=c\int_X\alpha^{2n}$, for all 
%$\alpha$, $\beta$ in $H^2(X)$ \cite{beauville}. 
If $\alpha$ belongs to $L$ and $c$ belongs to $H^{4n-2}(Y)$,
then $\int_X\alpha\pi^*(c)=P.D.(\alpha)\cap\pi^*(c)=\pi_*(P.D.(\alpha))\cap c
=0$. 
There is a class $\sigma_Y\in H^2(Y,\ComplexNumbers)$, such that
$\pi^*(\sigma_Y)$ spans $H^{2,0}(X)$, by Lemma
2.7 of \cite{kaledin-semi-small}.
Assume that $\beta$ belongs to $\pi^*H^2(Y)$.
The classes $\sigma^n\bar{\sigma}^{n-1}$, $\sigma^{n-1}\bar{\sigma}^n$,
and $\beta(\sigma\bar{\sigma})^{n-1}$, all belong to $\pi^*H^{4n-2}(Y)$. 
Thus $(\alpha,\beta)=0$, for all $\alpha\in L$. We conclude that 
$L$ is orthogonal to $\pi^*H^2(Y,\Integers)$.

The subspace $\pi^*H^2(Y,\RealNumbers)$ of 
$H^2(X,\RealNumbers)$ contains a positive definite
three-dimensional subspace spanned by the real part of
$H^{2,0}(X)\oplus H^{0,2}(X)\oplus \RealNumbers\pi^*\alpha$,
where $\alpha\in H^2(Y,\Integers)$ is the class of an ample line bundle.
$L$ is negative definite, since the Beauville-Bogomolov pairing 
on $H^2(X,\RealNumbers)$ has signature $(3,b_2(X)-3)$. 

We prove next the inclusion 
$H^2(U,\Integers)\subset L^\perp$.
Let $\tilde{\beta}\in H^2(X,\Integers)$ be a class, 
which belongs to the image of $H^2(U,\Integers)$. 
Then $\iota_{\widetilde{U}}^*(\tilde{\beta})=\phi^*(\beta)$, 
for a unique class $\beta\in H^2(U,\Integers)$. 
Let $\alpha\in H^2(X,\Integers)$ be the class Poincare-dual 
to an irreducible component $E_i$, of the
exceptional locus of $\pi$, of codimension one in $X$.
Set $B_i:=\pi(E_i)$, $B_i^0:=B_i\cap U$, and  
$E_i^0:=E_i\cap \widetilde{U}$. Then $E_i^0$
is a fibration $\phi_i:E_i^0\rightarrow B_i^0$ over $B_i^0$, 
with pure one-dimensional fibers, 
by Proposition \ref{prop-dissident-locus}. 
We get
\[
\int_X\alpha\tilde{\beta}\pi^*(\sigma_Y\bar{\sigma}_Y)^{n-1}=
\int_{E_i}\tilde{\beta}\pi^*(\sigma_Y\bar{\sigma}_Y)^{n-1}=
\int_{E_i^0}\phi_i^*\left([\beta(\sigma_Y\bar{\sigma}_Y)^{n-1})
\restricted{]}{B_i^0}\right).
\]
The above integral vanishes, since 
the restriction of the $2n-2$ form 
$\beta(\sigma_Y\bar{\sigma}_Y)^{n-1}$
to the $(2n-4)$-dimensional $B_i^0$ vanishes.
The vanishing of $(\alpha,\tilde{\beta})$ follows.
We conclude the inclusion
$H^2(U,\Integers)\subset L^\perp$.

The lattices $H^2(U,\Integers)$ and $L^\perp$ have the same rank,
by Lemma \ref{lemma-cohomology-of-X-versus-Y} part
\ref{lemma-item-corank-of-saturated-sublattice}. 
The equality $H^2(U,\Integers)= L^\perp$ follows, since 
$H^2(U,\Integers)$ is saturated in $H^2(X,\Integers)$, 
by Lemma \ref{lemma-cohomology-of-X-versus-Y} part \ref{lemma-item-saturated}.
\end{proof}

Let $\B$ be the set of connected components of $\Sigma\setminus\Sigma_0$. 
Let $\Lambda_B$, $B\in \B$, be the sublattice of 
$H^2(X,\Integers)$ 
spanned by the classes of the irreducible components of $\pi^{-1}(B)$.  
Let $\E_B$ be the set of irreducible components of the exceptional locus of
$\pi:X\rightarrow Y$, of pure co-dimension $1$, which dominate $B$.
Given $i\in \E_B$, denote by $E_i\subset X$ the corresponding divisor and let 
$e_i\in H^2(X,\Integers)$ be the class Poincare-dual to $E_i$. 
Set $E_i^0:=E_i\setminus \pi^{-1}\Sigma_0$. 
The fiber of $E_i^0\rightarrow B$ over a point $b\in B$ is a union of
smooth irreducible rational curves, which are all homologous in $X$. 
Denote by $e_i^\vee$ the class in $H^{4n-2}(X,\Integers)$
of such a rational curve. Consider $e_i^\vee$ also as a class in 
$H^2(X,\Integers)^*$, via Poincare-duality.

\begin{lem}
\label{lemma-L-B-is-G-invariant}
\begin{enumerate}
\item
\label{lemma-item-Lambda-B-is-G-invariant}
$\Lambda_B$ is a $G$-invariant sublattice of $H^2(X,\Integers)$. 
\item
\label{lemma-item-correspondence-inducing-g}
For every element $g\in G$, there exists an algebraic correspondence
$Z_g$ in $X\times_YX$, of pure dimension $2n$, inducing
the action of $g$ on $H^2(X,\Integers)$.
\item
\label{lemma-item-class-of-Z-g}
The endomorphism $[Z_g]_*:H^2(X,\Integers)\rightarrow H^2(X,\Integers)$,
induced by the correspondence $Z_g$, has the form
\[
[Z_g]_* \ \ \ = \ \ \ 
id + \sum_{B\in\B} \ \ \sum_{i,j\in\E_B} a_{ij}e_i\otimes e_j^\vee,
\]
for some non-negative integers $a_{ij}$.
\end{enumerate}
\end{lem}

\begin{proof}
Part \ref{lemma-item-Lambda-B-is-G-invariant} follows from 
part \ref{lemma-item-class-of-Z-g}.
We proceed to prove part \ref{lemma-item-correspondence-inducing-g}.
Fix an element $g\in G$. Let $C\subset Def(X)$ be a smooth 
connected Riemann surface, containing $0$, such that $C\setminus \{0\}$
is contained in the open subset $V$ introduced in Diagram 
(\ref{main-diagram}). Set $C^0:=C\setminus \{0\}$. 
Given a subset $S$ of $Def(X)$, let $\X_S$ be the restriction of $\X$ to $S$
and $\Y_S$ the restriction of $f^*\Y$ to $S$.
The morphism $\psi : \X_V\rightarrow V$ is $G$-equivariant with respect to
the $G$-action on $\X_V$ induced by the isomorphism
$\tilde{\nu}:\X_V\rightarrow \Y_V$. 
%The $G$-action extends to the open subset 
%of $\X$, where $\tilde{\nu}:\X\rightarrow f^*\Y$ is an isomorphism.
%Let $A$ be the open subset of $\X_C$, which is the complement of the 
%closed subset $\pi^{-1}(\Sigma)$ of the fiber $X$ of $\X_C$ over $0\in C$.
We get the isomorphism $\tilde{g}:\X_{C^0}\rightarrow (g^*\X)_{C^0}$. 
%which extends to $A$.
Let $\Gamma$ be the closure of the graph of the isomorphism $\tilde{g}$
in $\X_C\times_{C}(g^*\X)_C$. Then $\Gamma$ is contained in the inverse, via 
$\tilde{\nu}\times\tilde{\nu}$ of the diagonal in $\Y_C\times_C\Y_C$.
Hence, the fiber 
$Z_g:=\Gamma\cap [X\times X]$ is contained in $X\times_YX$. 
We consider $Z_g$ as a subscheme of $X\times X$.
The fiber $Z_g$ is of pure dimension $2n$, by the
upper-semi-continuity of fiber dimension, and the irreducibility of $\Gamma$.
Hence, the morphism $\Gamma\rightarrow C$ is flat.
The class $[Z_g]\in H^{4n}(X\times X)$  is thus the limit of the classes 
of the graph $\Gamma_t$ of the isomorphism 
$\tilde{g}_t:X_t\rightarrow X_{g(t)}$, $t\in C^0$. 
The limiting action $[Z_g]_*:H^2(X,\Integers)\rightarrow H^2(X,\Integers)$
is precisely the monodromy operator
$\gamma_g$, given in diagram (\ref{eq-gamma-g}).

Part \ref{lemma-item-class-of-Z-g} follows from Proposition
\ref{prop-dissident-locus} and the above construction of $Z_g$.
\end{proof}

%***********
\hide{
%***********
An isometry $g$ of 
$H^2(X,\Integers)$ is called a {\em monodromy operator}, 
if there exists a 
family $\X \rightarrow T$ (which may depend on $g$) 
of irreducible holomorphic symplectic manifolds, having $X$ as a fiber
over a point $t_0\in T$, 
and such that $g$ belongs to the image of $\pi_1(T,t_0)$ under
the monodromy representation. 
The {\em monodromy group} $Mon^2(X)$ of $X$ is the subgroup 
of $O[H^2(X,\Integers)]$ generated by all the monodromy operators. 
$Mon^2(X)$ was explicitly calculated in case 
$X$ is deformation equivalent to the Hilbert scheme $S^{[n]}$
of length $n$ subschmes of a $K3$ surface $S$
\cite{markman-constraints}, Theorem 1.2. 
%***********
%End \hide
%***********
}

Let $G_L\subset O[H^2(X,\Integers)]$ be the subgroup, 
which leaves invariant each of the sublattices $\Lambda_B$, $B\in \B$,  
of $H^2(X,\Integers)$, and 
acts as the identity on the sublattice $L^\perp$  orthogonal to $L$. 
%The following is an immediate corollary of Lemmas
%\ref{lemma-L} and \ref{lemma-L-B-is-G-invariant}. 

\begin{prop}
\label{prop-G-embedds-in-G-L}
\begin{enumerate}
\item
\label{cor-item-G-L-is-finite}
$G_L$ is a finite subgroup, 
which contains the image of the Galois group $G$, via the injective 
homomorphism 
$G\rightarrow O[H^2(X,\Integers)]$ constructed in Lemma
\ref{lemma-galois-cover}.
%\item
%\label{cor-L-G-vanishes}
%The $G$-invariant subspace $L^G$ vanishes.
\item
\label{G-is-a-pr-duct-of-Weyl-groups}
The group $G$ is isomorphic to a product of Weyl groups associated to 
root systems of finite type.% (a {\em crystallographic group}).
\end{enumerate}
\end{prop} 

\noindent
{\bf Caution:} Part \ref{G-is-a-pr-duct-of-Weyl-groups} of the proposition does
not relate the Weyl group factors of $G$ to the Weyl groups $W_B$, $B\in\B$,
in Theorem \ref{thm-G-is-isomorphic-to-product-of-weyl-groups}. 
At this point we do not claim even the non-triviality of $G$.

\begin{proof}
\ref{cor-item-G-L-is-finite})
Lemma \ref{lemma-L} and 
part \ref{lemma-item-Lambda-B-is-G-invariant} of Lemma 
\ref{lemma-L-B-is-G-invariant} reduce the proof 
to showing that $\gamma_g$ acts as the identity on
the subgroup $H^2(U,\Integers)$ of $H^2(X,\Integers)$, for all $g\in G$. 
%on $\pi^*H^2(Y)$.
%Fix $t\in V$, and let 
%$\iota_{f(t)}:Y_{f(t)}\hookrightarrow \restricted{\Y}{U}$ and 
%$\iota_U:\restricted{\Y}{U}\hookrightarrow \Y$ be the inclusions. 
%Denote by $H^2(Y_{f(t)},\RealNumbers)^{inv}$ the
%$\pi_1(U,f(t))$-invariant subspace.
%Then the image of the homomorphism
%$\iota_{f(t)}^*:H^2(\restricted{\Y}{U},\RealNumbers)
%\rightarrow H^2(Y_{f(t)},\RealNumbers)$ is 
%equal to $H^2(Y_{f(t)},\RealNumbers)^{inv}$, 
%by the Local Invariant Cycle Theorem (\cite{voisin-book-vol2}, Theorem 4.18). 
%Note that $Y$ is a deformation retract of $\Y$.
%The following is thus a factorization of $\pi^*:H^2(Y)\rightarrow H^2(X)$
%(with $\RealNumbers$-coeficients):
%\[
%H^2(Y)\cong 
%H^2(\Y)\RightArrowOf{\iota_U^*}H^2(\restricted{\Y}{U})
%\RightArrowOf{\iota_{f(t)}^*} 
%H^2(Y_{f(t)})^{inv} \IsomRightArrowOf{\nu_t^*}
%H^2(X_t)^G\cong H^2(X)^G.
%\]
%The invariance of $\pi^*H^2(Y)$ follows.
The action of $\gamma_g$ on $H^2(U,\Integers)$ is trivial, by
part \ref{lemma-item-class-of-Z-g} of Lemma \ref{lemma-L-B-is-G-invariant}. 
Indeed, let $Z_g\subset X\times_YX$ 
be the corresondence, of pure-dimension $2n$, inducing the action of 
$\gamma_g$. Let $\beta\in H^2(X,\Integers)$ be a class in the image of 
$H^2(U,\Integers)$. Using the notation of Lemma \ref{lemma-L-B-is-G-invariant}
we have $e_i^\vee(\beta)=0$, for all $i\in \cup_{B\in\B}\E_B$.
Hence, $[Z_g]_*(\beta)=\beta$.
%Then any irreducible component $Z'$ of $Z_g$, 
%different from the diagonal $\Delta_X$, satisfies $[Z']_*(\beta)=0$.

\ref{G-is-a-pr-duct-of-Weyl-groups})
The quotient $H^{1,1}(X,\ComplexNumbers)/G$ is smooth, by Lemma 
\ref{lemma-galois-cover} and Theorem \ref{thm-namikawa}.
$H^{1,1}(X,\ComplexNumbers)$ decomposes as the direct sum 
of $L_\ComplexNumbers:=L\otimes_\Integers\ComplexNumbers$ and 
$L_\ComplexNumbers^\perp \cap H^{1,1}(X,\ComplexNumbers)$,
since $L$ is negative-definite, by Lemma
\ref{lemma-L}. 
$G$ acts trivially on $L_\ComplexNumbers^\perp$, and hence 
$L_\ComplexNumbers/G$ is smooth as well. $G$ is thus a finite complex
reflection group, by 
\cite{bourbaki}, Ch. V, section 5, Theorem 4, and 
the statement follows from the classification 
of complex reflection groups preserving a lattice 
\cite{bourbaki}, Ch. VI, section 2, Proposition 9.
\end{proof}

Consider the case where $\pi:X\rightarrow Y$ is a small contraction.
Then $L=(0)$, both $G_L$ and $G$ are trivial, and $f:Def(X)\rightarrow Def(Y)$
is an isomorphism. The latter fact is already a consequence of a 
general result of Namikawa (see Proposition 
\ref{prop-namikawa-commutative-diagram} below). 

%****************************************************************
%
%****************************************************************
\subsection{The differential $df_0$}
%Let $Y$ be a $2n$-dimensional projective symplectic variety
%and $\pi:X\rightarrow Y$ a resolution, with $X$ a projective irreducible 
%holomorphic symplectic manifold.
Recall that the infinitesimal deformations of $Y$ are given by the
group $\Ext^1(\Omega^1_Y,\StructureSheaf{Y})$.
The morphism $\pi:X\rightarrow Y$ induces a natural morphism
$\pi_*:H^1(X,TX)\rightarrow \Ext^1(\Omega^1_Y,\StructureSheaf{Y})$,
which coincides with the differential $df_0$ of $f:Def(X)\rightarrow Def(Y)$
at $0$. Given $\xi\in H^1(X,TX)$, the class $\pi_*(\xi)$
is identified as follows.
Let $0\rightarrow \StructureSheaf{X}\rightarrow F \rightarrow \Omega^1_X
\rightarrow 0$ be a short exact sequence with extension class $\xi$.
We get the short exact sequence
\[
0\rightarrow \StructureSheaf{Y}\rightarrow \pi_*F \rightarrow \pi_*\Omega^1_X
\rightarrow 0,
\]
since $R^1_{\pi_*}(\StructureSheaf{X})$ vanishes, 
as $Y$ has rational singularities. 
Pulling back via $\pi^*:\Omega^1_Y\rightarrow \pi_*\Omega^1_X$,
we get the extension with class $\pi_*(\xi)$.

Let $\Sigma$ be the singular locus of $Y$ and 
$\Sigma_0$ the dissident locus, as in Proposition
\ref{prop-dissident-locus}.
Set $U:=Y\setminus \Sigma_0$ and $\widetilde{U}:=\pi^{-1}(U)$.

\begin{prop}
\label{prop-namikawa-commutative-diagram}
(\cite{namikawa-deformations}, Proposition 2.1) 
There is a commutative diagram
\begin{equation}
\label{namikawas-diagram}
\begin{array}{ccc}
H^1(X,TX)&\LongIsomRightArrow& H^1(\widetilde{U},T\widetilde{U})
\\
\pi_* \ \downarrow \ \hspace{2ex} & & \downarrow 
\\
\Ext^1(\Omega^1_Y,\StructureSheaf{Y})&\LongIsomRightArrow&
\Ext^1(\Omega^1_U,\StructureSheaf{U}),
\end{array}
\end{equation}
where the restriction horizontal maps are both isomorphisms.
\end{prop}

Let $H^{1,1}(X)=H^{1,1}(X)^G\oplus H^{1,1}(X)'$ be the $G$-equivariant
decomposition, where 
$
H^{1,1}(X)':=\oplus_{R\neq 1} \Hom_G[R,H^{1,1}(X)]\otimes R, 
$
and $R$ varies over all non-trivial irreducible complex representations of $G$.
Identify $H^1(X,TX)$ with $\Hom[H^{2,0}(X),H^{1,1}(X)]$.

\begin{lem}
\label{lem-differential}
The differential 
$df_0:\Hom[H^{2,0}(X),H^{1,1}(X)]\rightarrow 
\Ext^1(\Omega^1_Y,\StructureSheaf{Y})$ factors as the projection onto
$\Hom[H^{2,0}(X),H^{1,1}(X)^G]$, followed by an injective homomorphism
$
\Hom[H^{2,0}(X),H^{1,1}(X)^G]\rightarrow \Ext^1(\Omega^1_Y,\StructureSheaf{Y}).
$
\end{lem}

\begin{proof}
The statement follows immediately from Lemma \ref{lemma-galois-cover}, 
since $Def(Y)=Def(X)/G$.
\end{proof}

\begin{rem}
Once Theorem \ref{thm-G-is-isomorphic-to-product-of-weyl-groups} is proven, 
we would get that the invariant subspace $L^G$ of $L$ vanishes.
Hence, the rank of $df_0$ is equal to 
$\dim[H^2(U,\ComplexNumbers)\cap H^{1,1}(X)]=h^{1,1}(X)-\rank(L)$,
by Lemmas \ref{lemma-L} and \ref{lem-differential}.
\end{rem}

%****************************************************************
%
%****************************************************************
\subsection{Pro-representable deformation functors of open subsets}
Let $Art$ be the category of local Artin algebras over $\ComplexNumbers$,
and $Set$ the category of sets. Let
$D_X$, $D_Y$, $D_{\widetilde{U}}$, and $D_U$ be the 
deformation functors from $Art$ to $Set$, sending $A$ to the set of 
equivalence classes of deformations, of the corresponding variety, over $A$
\cite{schlessinger}, section (3.7). 
The terms {\em hull} and {\em pro-representable functors}
are defined in \cite{schlessinger}, Definition 2.7. 
The reader is refered to 
\cite{illusie}, \S{8.1}, for an excellent summery of basic definitions in
formal algebraic geometry. 
The following is an immediate corollary of Proposition
\ref{prop-namikawa-commutative-diagram}.

\begin{cor}
\label{cor-hull-of-deformation-functor-of-open-set}
The functors $D_X$, $D_{\widetilde{U}}$, $D_Y$ and $D_U$ are pro-representable.
Denote by $\DDef(X)$, $\DDef(Y)$, $\DDef(\widetilde{U})$, and $\DDef(U)$
the corresponding hulls. We have the commutative diagram
\[
\begin{array}{ccccc}
Def(X) & \LongLeftArrowOf{\iota_X} & \DDef(X) & \LongRightArrowOf{\rho_X} &
\DDef(\widetilde{U})
\\
f \ \downarrow \ \hspace{1ex} & & \hat{f} \ \downarrow \ \hspace{1ex} & & 
\hat{f}^\rho \ \downarrow \ \hspace{2ex}
\\
Def(Y)& \LongLeftArrowOf{\iota_Y} & \DDef(Y) & \LongRightArrowOf{\rho_Y} &
\DDef(U),
\end{array}
\]
where $\iota_X$ factors through an isomorphism of the hull 
$\DDef(X)$ with the completion
of the Kuranishi local moduli space $Def(X)$ at $0$, 
and similarly for $\iota_Y$.
The morphism $\rho_X$ is the one associated 
to the morphism of functors from $D_X$ to $D_{\widetilde{U}}$
induced by restriction. The morphism $\rho_X$ is an isomorphism. 
The morphism $\rho_Y$ is defined similarly and is an isomorphism. 
The morphism $\hat{f}$ is the completion of $f$
at $0$ and we set $\hat{f}^\rho:=\rho_Y\circ\hat{f}\circ \rho_X^{-1}$.
\end{cor}

\begin{proof}
The deformation functor $D_S:Art\rightarrow Set$, 
of a scheme $S$ over $\ComplexNumbers$, 
has a hull, if and only if $Ext^1(\Omega^1_S,\StructureSheaf{S})$
is finite dimensional, by \cite{schlessinger}, section 3.7.
Finite dimensionality, for $S=U$ or $\widetilde{U}$, is established
in Proposition \ref{prop-namikawa-commutative-diagram}.
$\DDef(X)$ and $\DDef(\widetilde{U})$ are pro-representable, since 
$H^0(X,TX)$ and $H^0(\widetilde{U},T\widetilde{U})$ both vanish.

The sheaf  $TY$ of derivations of $\StructureSheaf{Y}$ is torsion free.
The vanishing of $H^0(U,TU)$ would thus imply that of $H^0(Y,TY)$, 
and would consequently prove that 
$D_U$ and $D_Y$ are also pro-representable.
Hence, it suffices to prove 
that there is a sheaf isomorphism 
$\pi_*T\widetilde{U}\rightarrow TU$, yielding 
an isomorphism $H^0(\widetilde{U},T\widetilde{U})\cong H^0(U,TU)$. 
The spaces of global sections are the same, regardless if we use the 
Zariski or analytic topology. 
The sheaf-theoretic question is thus local,
in the analytic topology, and so reduces to the case of 
simple surface singularities
(\cite{reid}, section (3.4)).
The latter is proven in \cite{burns-wahl}, Proposition 1.2.
%Any global derivation of $U$
%preserves the ideal of the singular locus of $U$, and
%hence lifts to a derivation of the blow-up. Now $\widetilde{U}$
%is an iterated blow-up of $U$, so the vanishing of 
%$H^0(\widetilde{U},T\widetilde{U})$ implies the vanishing of $H^0(U,TU)$.

%The morphism of functors from $D_X$ to $D_{\widetilde{U}}$ is an isomorphism, 
%by \cite{artin-book}, Proposition 9.2.
%Hence, $\rho_X$ is an isomorphism. 
The same argument shows that each of $\rho_X$ and $\rho_Y$ is an isomorphism.
We prove it only for $\rho_Y$. 
Denote by $R_Y$ and $R_U$ the formal coordinate rings
of $\DDef(Y)$ and $\DDef(U)$. 
$R_Y$ is a formal power series ring, by Theorem
\ref{thm-namikawa}.
%The morphism $\rho_Y$ is a closed immersion, i.e., 
%$\rho_Y^*:R_U\rightarrow R_Y$ is surjective, 
%since $\rho_Y^*$ induces an isomorphism of Zariski cotangent 
%spaces, by Proposition \ref{prop-namikawa-commutative-diagram}.
Let $m$ be the maximal ideal of $R_U$.
$R_U$ is a quotient of the formal power series ring
$\ComplexNumbers[[m/m^2]]$ with cotangent space $m/m^2$, by construction 
(\cite{schlessinger}, Theorem 2.11).
The composition 
$\ComplexNumbers[[m/m^2]]\LongRightArrowOf{q} 
R_U\LongRightArrowOf{\rho_Y^*} R_Y$
is an isomorphism, as it induces an isomorphism of Zariski cotangent 
spaces, by Proposition \ref{prop-namikawa-commutative-diagram}. 
The homomorphism $\rho^*_Y$ is an isomorphism, 
since the homomorphism $q$ is surjective.
\end{proof}

%***********
% Hide
%***********
\hide{
\begin{rem}
The morphism $\hat{f}^\rho:\DDef(\widetilde{U})\rightarrow \DDef(U)$
is the classifying morphism associated to the formal deformation
$(\iota_Y\circ\hat{f}\circ \rho_X^{-1})^*[\Y\setminus \Sigma_0]$ of $U$ over
$\DDef(\widetilde{U})$. 
Set $\widetilde{\UU}:=(\iota_X\circ \rho_X^{-1})^*[\X\setminus 
\pi^{-1}(\Sigma_0)]$, and 
$\UU:=(\iota_Y\circ\rho_Y^{-1})^*[\Y\setminus\Sigma_0]$.
The completion of the morphism $\nu:\X\rightarrow Y$
restricts to a morphism $\hat{\nu}$ fitting in a commutative diagram
\[
\begin{array}{ccc}
\widetilde{\UU}&\LongRightArrowOf{\hat{\nu}}& \UU
\\
\downarrow & & \downarrow
\\
\DDef(\widetilde{U}) & \LongRightArrowOf{\hat{f}^\rho}&
\DDef(U)
\end{array}
\]
analogous to Diagram (\ref{diagram-f-general}).
\end{rem}
%***********
% End \hide
%***********
}

Let $g$ be an automorphism of $Def(X)$, such that $g(0)=0$.
We use the above Corollary to obtain a sufficient criterion
for $g$ to belong to the Galois group $G$  of $f:Def(X)\rightarrow Def(Y)$. 
Given a Riemann surface $C$ in $Def(X)$, containing $0$, denote by 
$\Y_C$ the restriction of $f^*\Y$ to $C$.
Set $\Y^0_C:=\Y_C\setminus \Sigma_0$, 
so that the fiber of $\Y^0_C$ over $0$ is $U:=Y\setminus\Sigma_0$. 
Let $V$ be the open subset of $Def(X)$ in Diagram (\ref{main-diagram}).
Denote by $\widehat{C}$ the completion of $C$ at $0$.

\begin{lem}
\label{lem-criterion-for-automorphism-to-be-deck-trans}
Assume that there exists a Zariski dense open subset $\TT$ of
$\PP[T_0Def(X)]$ satisfying the following property. 
For every connected  
Riemann surface $C\subset Def(X)$,
containing $0$, whose tangent line $T_0C$
belongs to $\TT$, and such that $C\setminus\{0\}$ is contained in $V$, 
the completions of $\Y^0_C$ and $g^*(\Y^0_{g(C)})$ along $U$
are equivalent formal deformations of $U$ over
$\widehat{C}$.
%there exists an isomorphism 
%\[
%\tilde{g} : \Y^0_C\rightarrow g^*(\Y^0_{g(C)}),
%\]
%of families over $C$, restricting to the identity along the fiber $U$.
Then $g$ belongs to $G$.
\end{lem}

\begin{proof}
It suffices to prove the equality 
$\restricted{f\circ g}{C}=\restricted{f}{C}$, 
for every curve $C$ as above, where $\restricted{f}{C}:C\rightarrow Def(Y)$
is the restriction of $f$ to $C$.

Let $\widehat{\Y}^0_C$ be the completion of $\Y^0_C$ along $U$.
Let $\restricted{\hat{f}}{C}:\widehat{C}\rightarrow \DDef(U)$ be 
the composition of the completion of $\restricted{f}{C}$ at $0$ 
with the isomorphism $\DDef(Y)\cong \DDef(U)$ of Corollary 
\ref{cor-hull-of-deformation-functor-of-open-set}.
Now $\restricted{\hat{f}}{C}$ is the classifying morphism of
$\widehat{\Y}^0_C$ and $\restricted{\widehat{f\circ g}}{C}$
is the classifying morphism of 
$\restricted{\hat{g}}{\widehat{C}}^*(\widehat{\Y}^0_{g(C)})$.
%The isomorphism $\tilde{g}$ induces an equivalence of 
These two formal deformations over $\widehat{C}$ are equivalent,
by assumption. 
The equality
$\restricted{\hat{f}}{C}=\restricted{\widehat{f\circ g}}{C}$ thus follows 
from the pro-representability of the functor $D_U$. The desired equality 
$\restricted{f\circ g}{C}=\restricted{f}{C}$ follows, as $C$ is connected.
\end{proof}

%****************************************************************
%
%****************************************************************
\subsection{Galois reflections via flops}
Let $E\subset X$ be the exceptional locus of $\pi:X\rightarrow Y$, 
and $E_i$ an irreducible component of $E$, which intersects $\widetilde{U}$
along a $\PP^1$ bundle $E^0_i\rightarrow \widetilde{B}_i$ over
an unramified cover $\widetilde{B}_i\rightarrow B$ of 
a connected component $B$ of $\Sigma\setminus\Sigma_0$.
Let $e_i\in H^2(X,\Integers)$ be the class Poincare-dual to $E_i$.
Let $e^\vee_i\in H^{4n-2}(X,\Integers)$ be the class Poincare-dual to
the fiber of $E_i$ over a point of $\widetilde{B}_i$. 

\begin{lem}
\label{lem-galois-reflection-via-flop}
The Beauville-Bogomolov degree $(e_i,e_i)$ is negative.
The isomorphism 
$H^2(X,\RationalNumbers)\rightarrow H^{4n-2}(X,\RationalNumbers)$,
induced by the Beauville-Bogomolov pairing, maps the 
class $\frac{-2e_i}{(e_i.e_i)}$ to the class $e^\vee_i$. Consequently,
the reflection
\[
g(x) \ \ := \ \ x - \frac{2(x,e_i)}{(e_i,e_i)}e_i
\]
has integral values, for all $x\in H^2(X,\Integers)$. 
The integral Hodge-isometry $g$
induces an automorphism of $Def(X)$, which belongs to the Galois group 
$G$.
\end{lem}

\begin{proof}
We follow the strategy suggested by Lemma 
\ref{lem-criterion-for-automorphism-to-be-deck-trans}.

\underline{Step 1:} (A generic one-parameter deformation).
Let $C\subset Def(X)$ be a connected Riemann surface containing $0$ and 
$\psi_C:\X_C\rightarrow C$ the restriction of the semi-universal family.
Assume that $C\setminus \{0\}$ is contained in the open subset $V$ of $Def(X)$,
given in Diagram (\ref{main-diagram}).
Set $\widetilde{\Sigma}_0:=\pi^{-1}(\Sigma_0)$ and 
$\X^0_C:=\X_C\setminus \widetilde{\Sigma}_0$, so that $\widetilde{U}$ is 
the fiber of $\psi^0:\X^0_C\rightarrow C$ over $0\in C$. 

Denote by 
$\epsilon\in 
\Ext^1_{\widetilde{U}}(N_{\widetilde{U}/\X^0},T\widetilde{U})\cong 
H^1(\widetilde{U},T\widetilde{U})$
the Kodaira-Spencer class and $\restricted{\epsilon}{E}$ its
restriction to
$H^1(E^0_i,\restricted{T\widetilde{U}}{E^0_i})$. 
Let $j:\restricted{T\widetilde{U}}{E^0_i}\rightarrow N_{E^0_i/\widetilde{U}}$
be the natural homomorphism, and 
$\alpha:=j_*(\restricted{\epsilon}{E})$ the pushed-forward class.
Denote by $\PP^1_t$, $t\in \widetilde{B}_i$, the fiber of $E^0_i$ over $t$.
Note that $\alpha$ restricts to $\PP^1_t$ as a class in 
$H^1(\PP^1_t,\omega_{\PP^1_t})$. 

$H^1(\widetilde{U},T\widetilde{U})$
is isomorphic to $H^1(X,TX)$, by \cite{namikawa-deformations},
Proposition 2.1, and to $H^{1,1}(X)$, via the holomorphic symplectic form.
Let $\TT$ be the complement in $\PP{H}^1(X,TX)$ of the 
projectivization of the kernel of the composition
$H^1(X,TX)\rightarrow H^{1,1}(X) \rightarrow H^{1,1}(\PP^1_t)$.
$\TT$ is non-empty, since the homomorphism is surjective.
Choose the curve $C$ so that $T_0C$ belongs to $\TT$. Then the restriction 
$\restricted{\alpha}{\PP^1_t}$ is non-trivial.

\underline{Step 2:} (A correspondence $\overline{Z}\subset X\times X$).
Let $\Delta_{\widetilde{U}}$ be the diagonal in 
$\widetilde{U}\times \widetilde{U}$ and set 
$Z:=\Delta_{\widetilde{U}}\cup [E^0_i\times_{\widetilde{B}_i} E^0_i]$.
Let $\overline{Z}$ be the closure of $Z$ in $X\times X$ and denote
by $\overline{Z}_*:H^2(X,\Integers)\rightarrow H^2(X,\Integers)$
the homomorphism induced by the correspondence. 
Let $(e^\vee_i)^\perp\subset H^2(X,\Integers)$ be the sublattice
annihilated by cup product with the class of $\PP^1_t$.
We clearly have that $\int_Xe_ie^\vee_i=
\deg\left(\omega_{\PP^1_t}\right)=-2$, 
$\overline{Z}_*(e_i)=-e_i$, and 
$\overline{Z}_*$ restricts to $(e^\vee_i)^\perp$ as the identity.
Hence, $\overline{Z}_*$ is an involution. 

\underline{Step 3:} (The homomorphism $\overline{Z}_*$ as a limit).
Our choice of the family $\X_C\rightarrow C$ is such, that 
$N_{E^0_i/\X^0}$ restricts to $\PP^1_t$ as a non-trivial extension of 
$\StructureSheaf{\PP^1_t}$ by $\omega_{\PP^1_t}$. 
There is a unique such extension, and so the restriction of $N_{E^0_i/\X^0}$
is isomorphic to 
$\StructureSheaf{\PP^1_t}(-1)\oplus\StructureSheaf{\PP^1_t}(-1)$.
Let $\widehat{\X}$ be the blow-up of $\X^0_C$ centered at $E^0_i$, 
and $\widehat{E}$ the exceptional divisor of $\widehat{\X}\rightarrow \X_C^0$. 
Then 
$\widehat{E}$ is a $[\PP^1\times\PP^1]$ bundle over $\widetilde{B}_i$.
Furthermore, $\StructureSheaf{\widehat{\X}}(\widehat{E})$ 
restricts as $\StructureSheaf{\PP^1\times\PP^1}(-1,-1)$ to each 
fiber\footnote{We have the three relations
$\omega_{\widehat{\X}}\cong\omega_{\X^0_C}(\widehat{E})$,
$\left(\omega_{\widehat{\X}}\restricted{\right)}{\widehat{E}}\cong
\omega_{\widehat{E}}\otimes\StructureSheaf{\widehat{E}}(-\widehat{E})$, and
$\omega_{\widehat{E}}\cong\omega_{\widehat{E}/E^0_i}\otimes\omega_{E^0_i}$.
Back substitution yields
$\StructureSheaf{\widehat{E}}(2\widehat{E})\cong
\omega_{\widehat{E}/E^0_i}\otimes\omega_{E^0_i}$.}.
Hence, there exists a $C$-morphism $\widehat{\X}\rightarrow \X'$ 
contracting $\widehat{E}$ along the second ruling 
\cite{artin-algebraization,fujiki-nakano}. 
Note that the fiber of $\psi':\X'\rightarrow C$ over $0\in C$ is 
naturally isomorphic to $\widetilde{U}$. 

The morphism 
$\widehat{\X}\rightarrow \X^0_C\times \X'$ is an embedding and we denote its 
image by  $\widehat{\X}$ as well. Then 
$\widehat{\X}\cap [\widetilde{U}\times\widetilde{U}]=Z$. 
The restriction homomorphism 
$H^2(X,\Integers)\rightarrow H^2(\widetilde{U},\Integers)$ is 
an isomorphism, by Lemma \ref{lemma-cohomology-of-X-versus-Y}
part \ref{lemma-item-iota-U-tilde-is-an-isom}.
We get that both $R^2_{\psi^0_*}(\Integers)$ and $R^2_{\psi'_*}(\Integers)$
are local systems over $C$.
Each of the two projections, from the correspondence $\widehat{\X}$ to
$\X^0_C$ and $\X'$, is a proper morphism. 
Thus, the correspondence $\widehat{\X}$
induces a homomorphism
\begin{equation}
\label{eq-X-widehat}
\widehat{\X}_* \ : \ R^2_{\psi^0_*}(\Integers) \ \ \ \longrightarrow \ \ \ 
R^2_{\psi'_*}(\Integers), 
\end{equation}
which is clearly an isomorphism of weight $2$
variations of Hodge structures. 
%It induces isomorphisms 
%$\psi^0_*\Omega^2_{\psi^0}\rightarrow \psi'_*\Omega^2_{\psi'}$,
%$R^1_{\psi^0_*}\Omega^1_{\psi^0}\rightarrow R^1_{\psi'_*}\Omega^1_{\psi'}$,
%and thus also $R^1_{\psi^0_*}T_{\psi^0}\rightarrow R^1_{\psi'_*}T_{\psi'}$,
%since $R^1_{\psi^0_*}T_{\psi^0}$ is isomorphic to 
%$\Hom[\psi^0_*\Omega^2_{\psi^0},R^1_{\psi^0_*}\Omega^1_{\psi^0}]$.
%We get the isomorphism 
%\[
%id_{TC}\otimes \widehat{\X}_* \ : \
%\Hom[TC,R^1_{\psi^0_*}T_{\psi^0}] \ \ \ \rightarrow \ \ \ 
%\Hom[TC,R^1_{\psi'_*}T_{\psi'}].
%\]
Similarly, we get an induced homomorphism
$Z_*:H^2(\widetilde{U},\Integers)\rightarrow H^2(\widetilde{U},\Integers)$,
which is conjugated to $\overline{Z}_*$ via the restriction isomorphism.
We conclude that the involution $\overline{Z}_*$
of $H^2(X,\Integers)$ is the limit of isometries, and is hence an isometry.
%and an element of $G_L$.

\underline{Step 4:} We prove that $\overline{Z}_*$ is 
the reflection by $e_i$. 
The degree of $e_i$ is negative, by Lemma \ref{lemma-L}.
Let $e_i^\perp\subset H^2(X,\Integers)$ be the orthogonal complement, 
of the class $e_i$, with respect to the Beauville-Bogomolov form. 
We have seen that the involution $\overline{Z}_*$
sends $e_i$ to $-e_i$ and acts as the identity on $(e^\vee_i)^\perp$. 
Hence $(e^\vee_i)^\perp=e_i^\perp$, 
since $\overline{Z}_*$ acts as an isometry and the eigenspaces of an isometry 
are pairwise orthogonal. 
Hence, $\overline{Z}_*$ is the reflection by $e_i$.
On the other hand, $\overline{Z}_*(x)=x+(e^\vee_i,x)e_i$, for all 
$x\in H^2(X,\Integers)$, by definition of $\overline{Z}$.
The equality $(e^\vee_i,\bullet)=\frac{-2(e_i,\bullet)}{(e_i,e_i)}$
follows.

\underline{Step 5:}
Let $g$ be the automorphism of $Def(X)$, induced by the Hodge-isometric 
reflection $\overline{Z}_*$, and $\hat{g}$ its completion at $0$.
We prove next the equivalence of two formal deformations
of $\widetilde{U}$: One is the completion of 
$\psi':\X'\rightarrow C$ along the fiber $\widetilde{U}$
and the other is obtained similarly from the pullback 
$g^*(\X^0_{g(C)}):=C\times_{g(C)}\X^0_{g(C)}$ 
to $C$
via $g$ of the ``family'' $\X^0_{g(C)}\rightarrow g(C)$.

Let $\widehat{C}$ be the completion of $C$ at zero.
The restrictions of $\X^0_C$ and $\X'$ to $\widehat{C}$ induces two classifying
maps, $\kappa_{\X^0}$ and $\kappa_{\X'}$, from $\widehat{C}$ to the hull 
$\DDef(\widetilde{U})$ of the deformation functor $D_{\widetilde{U}}$. 
The isomorphism (\ref{eq-X-widehat}) of variations of Hodge structures
implies that the outer square of the following diagram commutes.
\[
\begin{array}{ccccccc}
\widehat{C} & 
\LongRightArrowOf{\kappa_{\X^0}} & \DDef(\widetilde{U}) & 
\LongRightArrowOf{\iota_X\circ \rho_X^{-1}} &
Def(X) & \LongRightArrowOf{p} & \Omega 
\\
= \ \downarrow \ \hspace{2ex} & & \hat{g} \ \downarrow \ \hspace{1ex} & &
g \ \downarrow \ \hspace{1ex} & & \overline{Z}_* \ \downarrow \ \hspace{3ex}
\\
\widehat{C} & 
\LongRightArrowOf{\kappa_{\X'}} & \DDef(\widetilde{U}) & 
\LongRightArrowOf{\iota_X\circ \rho_X^{-1}} &
Def(X) & \LongRightArrowOf{p} & \Omega
\end{array}
\]
The right and middle squares commute, by definition. 
The period map $p$ is an embedding, by the Local Torelli Theorem 
\cite{beauville}.
Hence, the left square commutes as well, and we get the equality 
$\kappa_{\X'}=\hat{g}\circ\kappa_{\X^0}$.
Consequently, the completion along $\widetilde{U}$ of 
$\psi':\X'\rightarrow C$ 
is equivalent to that of $g^*(\psi^0):g^*(\X^0_{g(C)})\rightarrow C$
as a formal deformation of $\widetilde{U}$ over $\widehat{C}$.

\underline{Step 6:} 
A {\em modification} of a complex analytic space is a commutative diagram
\begin{equation}
\label{eq-modification}
\begin{array}{ccc}
\widetilde{U} & \subset & \widetilde{A}
\\
\pi \ \downarrow \ \hspace{2ex} & & \hspace{2ex} \ \downarrow \ \nu
\\
U & \subset & A,
\end{array}
\end{equation}
where $U$ and $\widetilde{U}$ are closed analytic subspaces
of $A$ and $\widetilde{A}$, the morphisms $\nu$ 
and $\pi$ are proper and surjective, 
and $\nu$ restricts as an isomorphism from 
$\widetilde{A}\setminus\widetilde{U}$ onto $A\setminus U$. 
We construct next a modification of $\widetilde{A}:=\X'$
to obtain a space $\Y':=A$. 
There is also a notion of a {\em formal modification} 
\cite{ancona-tomassini}, Ch I Section \S{2}.
We will only use the fact that the formal completion
of Diagram (\ref{eq-modification}) along $U$ and $\widetilde{U}$
is a formal modification. The following is an anaytic version 
of an algebraic result of M. Artin
\cite{artin-algebraization}.

\begin{thm}
\label{thm-C}
(\cite{ancona-tomassini}, Theorem C)
Let $\widetilde{U}$ be a closed analytic subspace of an analytic space 
$\widetilde{A}$, and $\widetilde{\AA}$ the formal completion of 
$\widetilde{A}$ along $\widetilde{U}$. Suppose that there exists a formal 
modification 
\begin{equation}
\label{eq-formal-modification}
\begin{array}{ccc}
\widetilde{U} & \subset & \widetilde{\AA}
\\
\pi \ \downarrow \ \hspace{2ex} & & \hspace{2ex} \ \downarrow \ \hat{\nu}
\\
U & \subset & \AA,
\end{array}
\end{equation}
with the additional hypothesis that $\AA$ is locally the formal completion
of an analytic space along a closed analytic subset. Then there exist
an analytic space $A$, containing $U$ as a closed analytic subset, 
and a modification (\ref{eq-modification}),
such that $\hat{\nu}$ is the completion of $\nu$ along $\widetilde{U}$.
The pair $(A,\nu)$ is unique, up to an isomorphism. 
\end{thm}

Given a subvariety $S$ of $Def(X)$, containing $0$, 
let $\Y_S$ be the restriction to $S$ of $f^*\Y$ and 
$\Y^0_S:=\Y_S\setminus\Sigma_0$, so that
the fiber of $\Y^0_S$ over $0\in S$ is $U:=Y\setminus\Sigma_0$. 

\begin{claim} 
\label{claim-existence-of-a-contraction}
There exists a normal analytic space $\Y'$,
admitting a morphism $\bar{\psi}':\Y'\rightarrow C$, 
and a proper surjective $C$-morphism 
$\nu':\X'\rightarrow \Y'$, having the following properties:
1) $\nu'$  restricts to the fibers over 
$0\in C$ as $\pi:\widetilde{U}\rightarrow U$. 2) $\nu'$ 
is an isomorphism over $\Y'\setminus U_{sing}$.
3) The completion of $\bar{\psi}'$ along $U$ is equivalent to
that of $g^*(\Y^0_{g(C)})\rightarrow C$.
\end{claim}

\begin{proof}
%We need to prove the existence of a small contraction $\nu'$, which
%exceptional locus is contained 
%in the fiber $\widetilde{U}$, and which restricts to the divisorial
%contraction $\pi$ along the fiber. 
%The existence of such a contraction, in the analytic category, 
%depends only on the formal completion of $\X'$ along the fiber
% $\widetilde{U}$ \cite{artin-algebraization,ancona-tomassini}. 
The formal completion of $\X'$ along the fiber $\widetilde{U}$ 
was shown to be isomorphic to the completion of
$g^*(\X^0_{g(C)})$ along $\widetilde{U}$, via a $\hat{C}$-morphism.
We may thus identify the two completions and denote both by 
$\widetilde{\AA}$.  
Let $\tilde{\nu}:\X\rightarrow f^*\Y$
be the natural lift of $\nu$, given by the universal property of the
fiber product $f^*\Y:=\Def(X)\times_{\Def(Y)}\Y$. 
Now $\tilde{\nu}$ restricts to a contraction
$g^*(\X^0_{g(C)})\rightarrow g^*(\Y^0_{g(C)})$.
Let $\hat{\nu}:\widetilde{\AA}\rightarrow \AA$ be 
the formal completion of the latter contraction along
$\widetilde{U}$ and $U$. We obtain a formal modification
(\ref{eq-formal-modification}). 
We apply Theorem \ref{thm-C} 
with $\widetilde{A}:=\X'$ and
conclude the existence of the morphism $\nu':\X'\rightarrow \Y'$,
satisfying property 1). $\Y'$ is normal, since $g^*(\Y^0_{g(C)})$ is
(see \cite{hartshorne}, Ch. III, Lemma 9.12).
Property 2) follows from the smoothness of $\X'$ and $\Y'\setminus U_{sing}$,
and the fact that $\nu':\X'\setminus \widetilde{U}\rightarrow \Y'\setminus U$
is an isomorphism.
Property 3) is clear, by construction.
%This completes the proof of Claim \ref{claim-existence-of-a-contraction}. 
\end{proof}

\underline{Step 7:} 
We are now ready to apply Lemma 
\ref{lem-criterion-for-automorphism-to-be-deck-trans}.
The assumptions of Lemma 
\ref{lem-criterion-for-automorphism-to-be-deck-trans}
would be verified, once we prove the equivalence
of the completions along $U$ of $\Y^0_C$ and $\Y'$
(here we use property 3 in Claim \ref{claim-existence-of-a-contraction}). 
%Let $\restricted{\tilde{\nu}}{C}:\X^0_C\rightarrow \Y^0_C$ be the
%restriction of $\tilde{\nu}$.
The composite morphism $\widehat{\X}\rightarrow \X^0_C
\RightArrowOf{\tilde{\nu}}\Y^0_C$ is 
clearly also the composition of $\widehat{\X}\rightarrow \X'$ and 
a $C$-morphism $\nu'':\X'\rightarrow \Y^0_C$. 

Set
$
a:= \nu'\times\nu'': \X' \rightarrow \Y'\times \Y^0_C.
$
We show that $a$ maps $\X'$ onto the graph of a $C$-isomorphism between
$\Y'$ and $\Y^0_C$. 
The fiber 
$\Y_{f(t)}$ is smooth, for all $t$ in $C\setminus \{0\}$, by our choice of $C$.
The statement is thus clear over $C\setminus \{0\}$.  
The restriction of $a$ to the special fiber 
is the morphism $\pi\times \pi:\widetilde{U}\rightarrow U\times U$, 
which maps $\widetilde{U}$ onto the diagonal in $U\times U$.
The morphism $a$ is proper, so its image is supported by a closed
subvariety $\widehat{\Y}$, 
which maps bijectively onto each of $\Y'$ and $\Y_C^0$.
It remains to show that each of these bijective morphisms 
induces an isomorphism of the structure sheaves. 
This question is infinitesimal, and so 
we may pass to the algebraic category. Note that 
$\Y^0_C$ is normal, by a lemma of Hironaka
\cite{hartshorne}, Ch. III, Lemma 9.12. 
Each of the projections from $\widehat{\Y}$, to each of the factors
$\Y'$ and $\Y^0_C$, is a proper morphism, since $\nu'$ and $\nu''$ are.
$\widehat{\Y}$ is thus the graph of an isomorphism, by Lemma
\ref{lemma-graph-of-isomorphism}. This completes the proof of Lemma
\ref{lem-galois-reflection-via-flop}.
%Set 
%\[
%a:=\restricted{\tilde{\nu}}{C}\times g^*(\restricted{\tilde{\nu}}{g(C)})
%: \X_C^0\times g^*(\X_{g(C)}^0)\rightarrow  \Y_C^0\times g^*(\Y_{g(C)}^0).
%\]
%It suffices to show that $a$ maps $\widehat{\X}$ onto the graph of an
%equivalence between the families $\Y_C^0$ and $g^*(\Y_{g(C)}^0)$.
%The fiber 
%$\Y_{f(t)}$ is smooth, for all $t$ in $C\setminus \{0\}$, 
%by our choice of $C$.
%Then $a$ is an isomorphism over $C\setminus \{0\}$ and 
%the restriction of $a$ to the special fiber 
%is the morphism $\pi\times \pi:\widetilde{U}\times\widetilde{U}\rightarrow
%U\times U$, which maps the special fiber $\overline{Z}$ of
%$\widehat{X}$ to the diagonal in $U\times U$. 
%The morphism $a$ is proper, so its image is supported by a closed
%subvariety $\widehat{\Y}$, 
%which maps bijectively onto each of $\Y_C^0$ and $g^*(\Y_{g(C)}^0)$. 
%It remains to show that each of these bijections induces an isomorphisms
%of the structure sheaves. This question is infinitesimal, and so 
%we may pass to the algebraic category. Note that 
%both $\Y^0_C$ and $\Y^0_{g(C)}$ are normal, by 
%a lemma of Hironaka
%\cite{hartshorne}, Ch. III, Lemma 9.12. 
%Each of the projections from $\widehat{\Y}$, to each of the factors
%$\Y^0_C$ and $\Y^0_{g(C)}$, is 
%a proper morphism, since $a$ is.
%$\widehat{\Y}$ is thus the graph of an isomorphism, by Lemma
%\ref{lemma-graph-of-isomorphism}. 
\end{proof}

\begin{lem}
\label{lemma-graph-of-isomorphism}
Let $Y_1$, $Y_2$ be normal schemes of finite type over $\ComplexNumbers$
and $\widehat{Y}\subset Y_1\times Y_2$ a closed integral subscheme.
Assume that the projection $\pi_i:\widehat{Y}\rightarrow Y_i$
is a proper and bijective morphism, for $i=1,2$.
Then $\widehat{Y}$ is the graph of an isomorphism from $Y_1$ onto $Y_2$. 
\end{lem}

\begin{proof} (See the proof of 
\cite{hartshorne}, Corollary III.11.4)
The projection $\pi_i:\widehat{Y}\rightarrow Y_i$ is clearly a homeomorphism,
since $\pi$ is bijective continuous and closed. 
It suffices to prove that the sheaf homomorphism
$\StructureSheaf{Y_i}\rightarrow \pi_{i,*}\StructureSheaf{\widehat{Y}}$
is an isomorphism. The question is local, so we may assume that 
$Y_i=Spec(A_i)$ and $\widehat{Y}=Spec(B)$. $B$ is a finitely generated
$A_i$-module, since $\pi_{i,*}\StructureSheaf{\widehat{Y}}$ is a coherent
$\StructureSheaf{Y_i}$ module. $A_i$ and $B$ have the same quotient field,
and $A_i$ is integrally closed. Consequently $A_i=B$.
\end{proof}

Keep the notation of Step 1 of the proof of 
Lemma \ref{lem-galois-reflection-via-flop}.

\begin{cor}
\label{cor-galois-reflection-via-flop}
The complex manifold $g^*(\X^0_{g(C)}):=C\times_{g(C)}\X^0_{g(C)}$ 
is isomorphic to the flop of $\X^0_C$
along $E_i^0$ via a $C$-isomorphism.
\end{cor}

\begin{proof}
Let $\X'$ be the flop of $\X^0_C$
along $E_i^0$, constructed in Step 3 of the proof of Lemma
\ref{lem-galois-reflection-via-flop}.
We have constructed the following two modifications in 
the proof of Lemma \ref{lem-galois-reflection-via-flop}:
\[
\begin{array}{ccc}
\widetilde{U} & \subset & g^*(\X^0_{g(C)})
\\
\pi \ \downarrow \ \hspace{2ex} & & \hspace{2ex} \ \downarrow \ \tilde{\nu}
\\
U & \subset & g^*(\Y^0_{g(C)})
\end{array}
\ \ \ \mbox{and} \ \ \ 
\begin{array}{ccc}
\widetilde{U} & \subset & \X'
\\
\pi \ \downarrow \ \hspace{2ex} & & \hspace{2ex} \ \downarrow \ \nu'
\\
U & \subset & g^*(\Y^0_{g(C)}),
\end{array}
\]
the left in Step 6 and the right in Step 7.
Their completions, along $\widetilde{U}$ and $U$,
are isomorphic formal modifications, 
by Claim \ref{claim-existence-of-a-contraction}.
The desired isomorphism $\X'\cong g^*(\X^0_{g(C)})$ follows from the
uniqueness part of the statement of the existence and uniqueness
of dilitations \cite{ancona-tomassini}, Theorem D.
\end{proof}

%****************************************************************
%
%****************************************************************
\subsection{Weyl groups as Galois groups}
We prove Theorem \ref{thm-G-is-isomorphic-to-product-of-weyl-groups} 
in this section, provided that all connected components of
$\Sigma\setminus \Sigma_0$ satisfy 
\preprint{the technical Assumption
\ref{assumption-no-orbits-of-two-adjacent-nodes} introduced below.
The proof of the Theorem is completed in section 
\ref{sec-folded-A-even-case-of-the-theorem}, dropping the assumption.}
\journal{Assumption
\ref{assumption-no-orbits-of-two-adjacent-nodes} introduced below.
The proof of the Theorem in the remaining special cases 
is outlined in section 
\ref{sec-brief-description-of-folded-A-2k-case}.
A detailed proof can be found in Lemma 4.24 of 
the preprint version of this paper \cite{preprint-version}.}

Let $B$ be a connected component of $\Sigma\setminus \Sigma_0$.
Then $Y$ has singularities along $B$ of type
$A_r$, $D_r$, or $E_r$, 
by Proposition \ref{prop-dissident-locus}.
Set $E_B:=\pi^{-1}(B)$. Choose a point $b$ in $B$. 
We adopt, throughout this section, the following:

\begin{assumption}
\label{assumption-no-orbits-of-two-adjacent-nodes}
If $Y$ has $A_n$ singularities along $B$, and $n$ is even, then 
$\pi_1(B,b)$ acts trivially on the set of 
irreducible components of $\pi^{-1}(b)$.
%None of the $\pi_1(B,b)$-orbits, 
%of the irreducible components of $\pi^{-1}(b)$,
%consists precisely of two adjacent irreducible components.
\end{assumption}

Each irreducible component $E_i^0$ of $E_B$ 
is a $\PP^1$-bundle $\pi_i:E^0_i\rightarrow \widetilde{B}_i$ over an 
unramified cover $u_i: \widetilde{B}_i\rightarrow B$, such that 
$u_i\circ \pi_i$ is equal to the restriction of $\pi$ to $E_i^0$.
This follows from 
Assumption \ref{assumption-no-orbits-of-two-adjacent-nodes}, and the 
classification of the automorphism groups of Dynkin diagrams
(\cite{humphreys}, Section 12.2 Table 1).
%Note that if Assumption \ref{assumption-no-orbits-of-two-adjacent-nodes}
%fails, then one of the irreducible components $E^0_i$ is a fibration
%over $B$, whose fiber is the union of two smooth rational curves 
%meeting at one point.

Let $\tau_B$ be the type of the Dynkin diagram of the singularity
of $Y$ along $B$. Then $\tau_B$ is also the type of the Dynkin 
diagram of the fiber $\pi^{-1}(b)$. The fundamental group 
$\pi_1(B,b)$ acts on the dual graph, via graph automorphism.
Let $\bar{\tau}^\vee_B$ be the type of the Dynkin diagram 
obtained by folding the Dynkin diagram of the fiber, 
via the action of $\pi_1(B,b)$, 
following the procedure recalled in section \ref{sec-folding}. 
The dual type is denoted by $\bar{\tau}_B$.
Let $E_i$ be the closure of $E_i^0$ in $X$
and set $e_i$ to be its class in $H^2(X,\Integers)$. 
Let $e^\vee_i$ be the class in $H^{4n-2}(X,\Integers)$
Poincare-dual to the fiber of $E_i^0$
over a point in $\widetilde{B}_i$. 
Set $\Lambda_B:={\rm span}_\Integers\{e_1, \dots, e_{\bar{r}}\}$ and
$\Lambda_B^\vee:={\rm span}_\Integers\{e^\vee_1, \dots, e^\vee_{\bar{r}}\}$,
where $\bar{r}$ is the rank of the root system of type $\bar{\tau}_B$.

\begin{lem}
\label{lem-Weyl-group-action}
\begin{enumerate}
\item
\label{lem-item-cartan-matrix}
We can rearrange the irreducible components $E_i$, so that the
matrix $-\int_Xe^\vee_ie_j$ is a Cartan matrix\footnote{The Cartan matrix
of a root system with fundamental basis $\{e_1^\vee, \dots e_{\bar{r}}^\vee\}$
has entry $e_i^\vee(e_j)$ in the $i$-th row and $j$-th column.
%We get here the transposed matrix (of the dual root system)
Note that we are folding the root system of the fiber.} 
of type $\bar{\tau}^\vee_B$.
%corresponding to the type of the singularity of $Y$ along $B$. 
\item
\label{lem-item-rank}
The lattices $\Lambda_B$ and $\Lambda_B^\vee$ both have rank $\bar{r}$.
\item
\label{lem-item-dual-root-systems}
The isomorphism 
$H^2(X,\RationalNumbers)\rightarrow H^{4n-2}(X,\RationalNumbers)$,
induced by the Beauville-Bogomolov pairing, 
restricts to an isomorphism from $\Lambda_B\otimes_\Integers\RationalNumbers$
onto $\Lambda_B^\vee\otimes_\Integers\RationalNumbers$, 
mapping the class $\frac{-2e_i}{(e_i.e_i)}$ to the class $e^\vee_i$.
\item
\label{lem-item-reflection-group-is-weyl-group}
Let $W_B$ be the subgroup of $G$, generated by the reflections
with respect to $e_i$, as in Lemma \ref{lem-galois-reflection-via-flop}.
Then $W_B$ is isomorphic to the Weyl group of type $\bar{\tau}_B$.
\item
\label{lemma-item-W-is-product-of-W-B}
Let $\B_0\subset \B$ be the subset consisting of connected
components $B$ of $\Sigma\setminus\Sigma_0$
satisfying Assumption
\ref{assumption-no-orbits-of-two-adjacent-nodes}.
The subgroup $W$ of $G$, generated by $\cup_{B\in\B_0}W_B$,
is isomorphic to $\prod_{B\in \B_0}W_B$.
\item
\label{lemma-item-G-equal-W}
If $\B_0=\B$, then $G=\prod_{B\in \B}W_B$.
\end{enumerate}
\end{lem}

Given a lattice $M$ of maximal rank in 
$\Lambda_B\otimes_\Integers\RationalNumbers$, 
we denote by $M^*$ the sublattice of 
$\Lambda_B^\vee\otimes_\Integers\RationalNumbers$, consisting of 
classes $x$, such that $(x,y)$ is an integer, for all $y\in M$.
The {\em fundamental group of the root system} is, by definition,
the group $\Pi_B:=(\Lambda_B^*/\Lambda_B^\vee)$.

\begin{proof}
\ref{lem-item-cartan-matrix}) This part follows immediately from 
the definition of an ADE-singularity along $B$,
if $\pi_1(B,b)$ acts trivially on the dual graph of the fiber
$\pi^{-1}(b)$. 
More generally, let $\F:=\{F_1, \dots, F_r\}$ be the set of irreducible 
components of $\pi^{-1}(b)$.
Denote by $f_i$ the node, corresponding to $F_i$, of the dual graph.
We regard $\{f_1, \dots, f_r\}$ as a basis of simple roots of the root
lattice of type $\tau_B$, which bilinear pairing was 
recalled in Section \ref{sec-folding}.
Let $[F_j]\in H^{4n-2}(X,\Integers)$ be the cohomology class of $F_j$.
Let $\Gamma$ be the image of $\pi_1(B,b)$ in the automorphism group
of the graph dual to the fiber $\pi^{-1}(b)$.
Note that if $F_i$ and $F_j$ belong to the same
$\Gamma$-orbit, then $[F_i]=[F_j]$.

Let $E_i^0$ be an irreducible component of $E_B$ and
$F_{\tilde{i}}$ an irreducible component of the fiber of $E_i^0\rightarrow B$
over $b$. Let $E_j^0$ and $F_{\tilde{j}}$ be another such pair. 
Then $\int_X[E_j][F_{\tilde{i}}]=-2$, if $i=j$, and 
\[
\int_X[E_j][F_{\tilde{i}}]
=
-\sum_{f_{\tilde{k}}\in 
\Gamma\cdot f_{\tilde{j}}}(f_{\tilde{k}},f_{\tilde{i}}),
\ \ \mbox{if} \ \ i\neq j,
\] 
by Assumption \ref{assumption-no-orbits-of-two-adjacent-nodes}.
Thus $-\int_X[E_j][F_{\tilde{i}}]$ is equal to the right hand side
of the equation in Lemma \ref{lemma-folding}.
Hence, we can order the orbit $\F/\Gamma$, so that the matrix 
$-\left(\int_X[E_j][F_{\tilde{i}}]\right)$
is the Cartan matrix of the folded Dynkin diagram.

\ref{lem-item-rank})
The sets $\{e_1, \dots, e_{\bar{r}}\}$
and $\{e^\vee_1, \dots, e^\vee_{\bar{r}}\}$ are linearly independent,
since the rank of the Cartan matrix is $\bar{r}$.

\ref{lem-item-dual-root-systems})
This part was proven in Lemma \ref{lem-galois-reflection-via-flop},
which applies by Assumption
\ref{assumption-no-orbits-of-two-adjacent-nodes}.

\ref{lem-item-reflection-group-is-weyl-group}) 
Part \ref{lem-item-dual-root-systems} implies that 
the basis $\{e_1, \dots, e_{\bar{r}}\}$
of $\Lambda_B$, 
endowed with minus the Beauville-Bogomolov pairing,
is a basis of simple roots for a root system of type $\bar{\tau}_B$.

\ref{lemma-item-W-is-product-of-W-B}) 
The lattice $L$ is negative definite and the
sub-lattices $\Lambda_B$, $B\in \B$, are pairwise orthogonal, yielding an
orthogonal direct sum decomposition 
$L\otimes_\Integers\RationalNumbers=
\oplus_{B\in\B}\Lambda_B\otimes_\Integers\RationalNumbers$. $W_B$
acts on $\Lambda_B$ via the reflection representation and it acts trivially on 
$\Lambda_{B'}$, if $B\neq B'$. 

\ref{lemma-item-G-equal-W})
The inclusion $W\subset G$ follows from part 
\ref{lemma-item-W-is-product-of-W-B}.
%The equality $W=G$ would thus follow, once 
We prove the inclusion $G\subset W$. 
The inclusion $G\subset G_L$ was proven in Proposition 
\ref{prop-G-embedds-in-G-L}.
Let $G_B$ be the image of $G$ 
in $O(\Lambda_B)$ via the composition
$G\rightarrow G_L\rightarrow O(\Lambda_B)$.
It suffices to prove the inclusion $G_{B}\subset W_B$,
since $W=\prod_{B\in \B} W_B$.
The inclusion $G_{B}\subset W_B$ would follow, if we can prove that 
$G$ acts trivially on the fundamental group $\Pi_B$,
since $W_B$ is equal to the subgroup of $O(\Lambda_B)$, 
which acts trivially on $\Pi_B$, by 
\cite{humphreys}, Exercise 5 in Section 13.4. 
Fix $g\in G$. 
Lemma \ref{lemma-L-B-is-G-invariant} implies that $g$ acts on $\Lambda_B$
via the class $[Z_g]=I+\sum_{i=1}^r\sum_{j=1}^ra_{ij}e_i\otimes e_j^\vee$,
where $I$ is the identity linear transformation, 
and $a_{ij}$ are non-negative integers. 
Let $\lambda$ be an element of $\Lambda_B^*$. 
Then $g(\lambda)=\lambda+\sum_{i=1}^r\sum_{j=1}^ra_{ij}\lambda(e_i)e_j^\vee$,
which is congruent to $\lambda$ modulo $\Lambda_B^\vee$.
\end{proof}

\begin{defi}
\label{def-root-system-Lambda-B-reduced-case}
Let $B\in \B$ be a connected component satisfying Assumption 
\ref{assumption-no-orbits-of-two-adjacent-nodes}.
The root system $\Phi_B\subset \Lambda_B$ is the union of the
$W_B$-orbits of the classes $e_i$, $1\leq i \leq \bar{r}$, 
of the irreducible components of 
the closure of $\pi^{-1}(B)$.
The root system $\Phi_B^\vee\subset \Lambda_B^\vee$ is the union of 
$W_B$-orbits of the classes $e_i^\vee$, $1\leq i \leq \bar{r}$, 
of the irreducible components of the fiber $\pi^{-1}(b)$, $b\in B$.
\end{defi}

Root systems, as defined in section 9.2 of 
\cite{humphreys}, are {\em reduced}. See \cite{humphreys},
Section 12.2 Exercise 3, for the non-reduced root system of type $BC_n$. 
As a corollary of Lemma \ref{lem-Weyl-group-action} we get:

\begin{cor}
\label{corollary-dual-root-systems}
\begin{enumerate}
\item
\label{cor-item-dual-root-lattices}
$\Lambda_B$ is the root lattice of 
the root systems $\Phi_B\subset \Lambda_B$ and $\Lambda_B^\vee$ 
is the root lattice of the root system 
$\Phi_B^\vee \subset \Lambda_B^\vee$. Both root systems are reduced
and are dual to each other.
$\Phi_B$ has type $\bar{\tau}_B$. 
\item
\label{cor-item-direct-sum-decomposition-for-G2-F4-E8}
If the type $\bar{\tau}_B$ is $E_8$, $F_4$, or $G_2$,
then $H^2(X,\Integers)$ decomposes as the orthogonal direct sum
$\Lambda_B\oplus \Lambda_B^\perp$.
If, furthermore, $B=\Sigma\setminus\Sigma_0$, then
$\Lambda_B^\perp$ is equal to the subspace
%$\pi^*H^2(Y,\Integers)$.
$H^2(U,\Integers)$ of $H^2(X,\Integers)$, defined in 
Lemma \ref{lemma-cohomology-of-X-versus-Y} part \ref{lemma-item-saturated}.
\end{enumerate}
\end{cor}

\begin{proof}
\ref{cor-item-dual-root-lattices})
$\Phi_B$ spans $\Lambda_B$ and $\Phi_B^\vee$ spans $\Lambda_B^\vee$, 
by construction.
The foldings, of every finite root system of type ADE, 
by any subgroup of its automorphism group, 
produces a reduced root system,
by the classification of the automorphism groups of 
Dynkin diagrams of ADE type (see \cite{carter}, Section 13.3). 

\ref{cor-item-direct-sum-decomposition-for-G2-F4-E8})
If $\bar{\tau}_B$ is $E_8$, $F_4$, or $G_2$, then 
the fundamental group $\Pi_B$ of the root system is trivial,
so that the sublattice $\Lambda_B^\vee$ of $H^{4n-2}(X,\Integers)$ 
is isometric to $\Lambda_B^*$. 
The extension $0\rightarrow \Lambda_B\rightarrow H^2(X,\Integers)\rightarrow 
H^2(X,\Integers)/\Lambda_B\rightarrow 0$ is split, by 
the Poincare-Duality isomorphism:
\[
H^{2}(X,\Integers)\cong H^{4n-2}(X,\Integers)^*\rightarrow (\Lambda_B^\vee)^*
\cong \Lambda_B.
\]
The kernel of the above homomorphism is the annihilator of 
$\Lambda_B^\vee$, which is equal to $\Lambda_B^\perp$,
by Lemma \ref{lem-Weyl-group-action}
part \ref{lem-item-dual-root-systems}. 
If $\Sigma\setminus\Sigma_0=B$, then 
$\Lambda_B^\perp=H^2(U,\Integers)$, by 
Lemma \ref{lemma-L}.
\end{proof}

%********************************************************************
%
%********************************************************************
\subsubsection{The saturation of the root lattice $\Lambda_B$}
Let $L_B$ be the saturation of the sublattice $\Lambda_B$ 
of $H^2(X,\Integers)$.
Let $L_B^\vee$ be the saturation of the sublattice 
$\Lambda_B^\vee$ of $H^{4n-2}(X,\Integers)$.
Let $\Abs{\Pi_B}$ be the cardinality of the fundamental group $\Pi_B$
of the root system.
We clearly have the flag
\[
\Lambda_B^\vee\subset L_B^\vee \subset L_B^* \subset \Lambda_B^*.
\]
We get the equality
\[
%\label{eq-factorization-of-the-order-of-Gamma}
\Abs{L_B^\vee/\Lambda_B^\vee}\cdot \Abs{L_B/\Lambda_B} \cdot
\Abs{L_B^*/L_B^\vee}
\ \ \ = \ \ \ \Abs{\Pi_B}.
\]
$\Pi_B$ is determined by the type $\bar{\tau}_B$ of the folded root system 
as follows:\\
$A_r$: $\Integers/(r+1)\Integers$,  
$B_r$, $C_r$: $\Integers/2\Integers$, 
$D_r$: $\Integers/4\Integers$, if $r$ is odd, and 
$\Integers/2\Integers\times \Integers/2\Integers$, if $r$ is even,
$E_6$: $\Integers/3\Integers$, 
$E_7$: $\Integers/2\Integers$, 
$E_8$, $F_4$, $G_2$: trivial \cite{humphreys}, section 13.1.

%**********************
% \preprint rest of section
%**********************
\preprint{
The groups  $L_B/\Lambda_B$ and $L_B^\vee/\Lambda_B^\vee$ are invariants of 
the singularity, which need not be determined by the type of the 
folded root system, if the fundamental group $\Pi_B$ is non-trivial.

\begin{example}
For the singularity type $A_1$ we have three possibilities
$\Abs{L_B/\Lambda_B}=2$,
$\Abs{L_B^\vee/\Lambda_B^\vee}=2$, or
$\Abs{L_B^*/L_B^\vee}=2$.
Following are two examples.
%$L_B/\Lambda_B\cong\Integers/2\Integers$, the other with trivial
%$L_B/\Lambda_B$. 

a) Let $S$ be a $K3$ surface, $S^{[n]}$ the Hilbert scheme of length $n$ 
subschemes of $S$, $S^{(n)}$ the $n$-th symmetric product, and
$\pi:S^{[n]}\rightarrow S^{(n)}$ the Hilbert-Chow morphism.
Then $\pi$ is a contraction of type $A_1$, 
$B:=\Sigma\setminus\Sigma_0$ is irreducible, the exceptional divisor $E$
has class $e=2\delta$, where $\delta$ is integral, primitive, and 
$(\delta,\delta)=2-2n$ \cite{beauville}.
Hence, $L_B/\Lambda_B\cong\Integers/2\Integers$.

b) Let $\pi:X\rightarrow Y$ be a contraction of type $A_1$, 
such that the class $e$ of the exceptional divisor satisfies
$(e,e)=-2$, and $(e,x)=1$, for some $x\in H^2(X,\Integers)$. Then
$(e^\vee,x)=(e,x)$, so $\Abs{L_B/\Lambda_B}=\Abs{L_B^\vee/\Lambda_B^\vee}=1$,
and thus $\Abs{L_B^*/L_B^\vee}=2$.
The contraction of a $-2$ curve on a $K3$ surface $X$ provides such an 
example. Higher dimensional examples can be found in
\cite{yoshioka-examples-of-reflections}, or \cite{markman-reflections},
Corollary 3.19 in the case of Mukai vectors $v$ with $\chi(v)=0$.
\end{example}

%****************************************************************
%
%****************************************************************
\subsection{Weyl group actions via flops}

Let $B\in\B$ be a connected component satisfying Assumption 
\ref{assumption-no-orbits-of-two-adjacent-nodes}.
Let $g\in W_B$ and write $g$ as a word in the reflections with respect to 
the simple roots $\{e_1, \dots e_{\bar{r}}\}$
\[
g \ \ = \ \ \rho_{e_{i_k}}\rho_{e_{i_{k-1}}}\cdots \rho_{e_{i_1}}.
\]
Let $C\subset Def(X)$ be a connected Riemann surface containing $0$, such that
$C\setminus\{0\}$ is contained in the subset $V$ given in Diagram
\ref{main-diagram}. Denote by $\Phi_B\subset \Lambda_B$ the root system
(Definition \ref{def-root-system-Lambda-B-reduced-case}).
Assume further that $T_0C$ is not contained in the $+1$ 
eigenspace in $T_0Def(X)$ of any reflection with respect to 
a root in $\Phi_B$. 
Notice that the special fiber of the flop of $\X^0_C$ centered along
$E_i^0$, $1\leq i \leq \bar{r}$, is naturally identified with 
the fiber $\widetilde{U}$ of $\X^0_C$. Hence, we can talk about sequences of 
such flops. Furthermore, 
the following equality holds for every $h\in W_B$ and for every $E^0_i$:
\begin{equation}
\label{eq-pull-back-commutes-with-flops}
h^*(\mbox{flop of} \ \X^0_{h(C)} \ \mbox{along} \ E^0_i) \ \ = \ \ 
\mbox{flop of} \ h^*(\X^0_{h(C)}) \ \mbox{along} \ E^0_i.
\end{equation}

\begin{lem}
\label{lemma-Weyl-group-action-via-flops}
The complex manifold $g^*(\X^0_{g(C)})$ is isomorphic to the one 
obtained from $\X^0_C$ by the sequence of $k$ flops, starting with
$E^0_{i_1}$ and ending with $E^0_{i_k}$. 
\end{lem}

\begin{proof}
The proof is by induction on $k$. The case
$k=1$ was proven in Corollary \ref{cor-galois-reflection-via-flop}:
$\rho^*_{e_i}(\X^0_{\rho_{e_i}(C)})$ is isomorphic to the flop of $\X^0_C$ 
along $E^0_i$. 
Assume that $k\geq 2$ and the statement holds for $k-1$. 
Set $\tilde{g}:=\rho_{e_{i_{k-1}}}\cdots \rho_{e_{i_1}}$.
Then $\tilde{g}^*(\X^0_{\tilde{g}(C)})$ is obtained from 
$\X^0_C$ by the sequence of flops starting with $E^0_{i_1}$
and ending with $E^0_{i_{k-1}}$. 
Apply next the case $k=1$ of the Lemma with the Riemann surface $\tilde{g}(C)$
and the reflection $\rho_{e_{i_k}}$ to get\\
$g^*(\X^0_{g(C)})=
\tilde{g}^*\rho_{e_{i_k}}^*(\X^0_{\rho_{e_{i_k}}\tilde{g}(C)})
=\tilde{g}^*(\mbox{the flop of} \ \X^0_{\tilde{g}(C)} \ \mbox{along} \
E^0_{i_k})$\\
$\stackrel{(\ref{eq-pull-back-commutes-with-flops})}{=} 
\mbox{the flop of} \ \tilde{g}^*(\X^0_{\tilde{g}(C)}) \ \mbox{along} \ 
E^0_{i_k}.$
\end{proof}

\begin{example}
\label{example-A-2}
Choose a connected component $B\in\B$, a point $b\in B$, and assume
that the type $\tau_B$, of the Dynkin diagram of the fiber
$\pi^{-1}(b)$, is $A_2$, and that $\pi_1(B,b)$ acts trivially on 
the Dynkin graph, so that $\bar{\tau}^\vee_B=A_2$ as well.
The existence of such a component $B$, for some contraction
$\pi:X\rightarrow Y$, with $2n$-dimensional $X$ and $Y$, for all $n$, 
is exhibited in section \ref{sec-Hilbert-scheme-of-K3-resolving-ADE}.
Given such a component $B$, then $E_B:=\pi^{-1}(B)$ is the union
of two irreducible components
$E^0_1$ and $E^0_2$, such that $E^0_i$ is a $\PP^1$ bundle over $B$
with section $E^0_1\cap E^0_2$. 
Set $e_i:=[E_i]\in H^2(X,\Integers)$, $i=1,2$, as above, 
and let $\rho_{e_i}\in G$
be the element of the Galois group corresponding to the reflection with 
respect to $e_i$.
Note the equalities
\[
\rho_{e_1}\rho_{e_2}\rho_{e_1}=\rho_{e_1+e_2}=\rho_{e_2}\rho_{e_1}\rho_{e_2}.
\]
Choose a smooth connected Riemann surface $C\subset Def(X)$, as in Lemma 
\ref{lemma-Weyl-group-action-via-flops}.
We see that flopping $\X^0_C$ first along $E^0_1$, 
then along $E^0_2$, and finally along $E^0_1$ again, yields the same 
``family''
$\rho_{e_1+e_2}^*(\X^0_{\rho_{e_1+e_2}(C)})$, as
flopping first along $E^0_2$, then along $E^0_1$, and finally along $E^0_2$ 
again.
\end{example}

The above Example will enable us to complete the proof of
Theorem \ref{thm-G-is-isomorphic-to-product-of-weyl-groups},
dropping Assumption \ref{assumption-no-orbits-of-two-adjacent-nodes}. 
We will further need the following Lemma. Keep the notation of the Example.
Let $\Z\subset 
\left[\X^0_C\times \rho^*_{e_1+e_2}(\X^0_{\rho_{e_1+e_2}(C)})\right]$ 
be the closure of the graph of the composite isomorphism
\[
[\X^0_C\setminus \widetilde{U}]\cong
[\Y^0_C\setminus U]\cong 
[\rho_{e_1+e_2}^*(\Y^0_{\rho_{e_1+e_2}(C)})\setminus U]\cong
[\rho^*_{e_1+e_2}(\X^0_{\rho_{e_1+e_2}(C)})\setminus \widetilde{U}].
\]
Let $Z$ be the fiber of $\Z$ over $0\in C$.

\begin{lem}
\label{lemma-correspondence-for-A-2-flop}
The reduced induced subscheme structure of $Z$ 
consists of the union of the following five irreducible 
components: The diagonal $\Delta_{\widetilde{U}}$, and the four components
$E^0_1\times_{B}E^0_1$, $E^0_2\times_{B}E^0_2$, 
$E^0_1\times_{B}E^0_2$, $E^0_2\times_{B}E^0_1$
of $E_B\times_{B}E_B$.
$Z$ is generically reduced along each of these components.
\end{lem}

\begin{proof}
Choose the equality $\rho_{e_1+e_2}=\rho_{e_1}\rho_{e_2}\rho_{e_1}$ and
consider the corresponding sequence of flops
\[
\begin{array}{ccccccccccl}
&\Z_1 &&\Z_2 &&\Z_3
\\
& \swarrow  \hspace{1ex} \searrow & & \swarrow \hspace{1ex} \searrow & &  
\swarrow \hspace{1ex} \searrow 
\\
\X^0_C & &  \rho^*_{e_1}(\X^0_{\rho_{e_1}(C)}) & &  
\rho^*_{e_1}\rho^*_{e_2}(\X^0_{\rho_{e_2}\rho_{e_1}(C)}) & & 
\rho^*_{e_1+e_2}(\X^0_{\rho_{e_1+e_2}(C)}).
\end{array}
\]
Each $\Z_i$ embeds as a correspondence in the product of its two successive
blow-downs. $\Z$ is the composite correspondence
$\Z=\Z_3\circ\Z_2\circ\Z_1$. 
Let $Z_i$ be 
the fiber of $\Z_i$ over $0\in C$. For $i=1,2$ we have 
$Z_i:=\Delta_{\widetilde{U}}\cup \left[E^0_i\times_B E^0_i\right]$
and $Z_3=Z_1$.
$Z$ is generically equal to the composition $Z_3\circ Z_2 \circ Z_1$
of the three correspondences, which is the union of the five irreducible
components as stated.
\end{proof}

%Let $\overline{Z}$ the closure of $Z$ in $X\times X$.

%************************************************************
%
%************************************************************
\subsection{The folded $A_{2k}$ case of Theorem 
\ref{thm-G-is-isomorphic-to-product-of-weyl-groups}}
\label{sec-folded-A-even-case-of-the-theorem}
Drop Assumption \ref{assumption-no-orbits-of-two-adjacent-nodes}. 
Fix $B\in\B$, along which $Y$ has an $A_r$-singularity, $r$ even, 
a point $b\in B$, and assume that the image $\Gamma$ of $\pi_1(B,b)$, 
in the automorphism group of the graph of the fiber $\pi^{-1}(b)$, is
$\Integers/2\Integers$. 
%Let $\gamma\in\Gamma$ be the non-trivial element.
Let $\Lambda_B\subset H^2(X,\Integers)$ be the lattice generated by the 
classes $e_j$ of the irreducible components $E_j$ of the closure of $E_B$.
Define $\Lambda_B^\vee$ as the lattice spanned by the classes 
$[F_j]$, where $F_j$ is one of the irreducible components of
the fiber of $E_j^0$ over $b$.
We will see below that $\Lambda_B$ is a unimodular 
(non-reduced) root lattice of type $BC_{\bar{r}}$, $\bar{r}=r/2$, 
and $\Lambda_B^\vee$ is isometric to the dual lattice $\Lambda_B^*$
(Remark \ref{rem-BC}). 

We define a sublattice $\tilde{\Lambda}_B$ of $\Lambda_B$. 
The irreducible components $F_j$ of the fiber $\pi^{-1}(b)$
come with two natural orderings, one the reverse of the other.
The two are interchanged by the non-trivial element of $\Gamma$. 
Choose one of the orderings and let $E_j^0$ 
be the irreducible component of $E_B$ containing the fiber $F_j$, 
$1\leq j \leq \bar{r}$. Then $E_{\bar{r}}^0$ is the 
irreducible component, such that the fiber of $E_{\bar{r}}^0$
over $b$ consists of the pair of middle components $F_{\bar{r}}$, 
$F_{\bar{r}+1}$ in the graph of the fiber.
Set $\tilde{e}_j:=[E_j]$, if $j\neq \bar{r}$,  
$\tilde{e}_{\bar{r}}:=2[E_{\bar{r}}]$, and let
$\tilde{\Lambda}_B:={\rm span}\{\tilde{e}_1, \dots, \tilde{e}_{\bar{r}}\}$. 
We set $\tilde{e}_j^\vee:=[F_j]$, regardless if $j$ and $\bar{r}$ 
are equal or not. 
Set $\tilde{\Lambda}^\vee_B:=
{\rm span}\{\tilde{e}_1^\vee, \dots, \tilde{e}_{\bar{r}}^\vee\}$. 
Then $\tilde{\Lambda}^\vee_B=\Lambda_B^\vee$.
%Note that the intersection of the class $2[F_{\tilde{i}}]$, of the fiber 
%of $E_i^0$ over $b$, with each $[E_j]$ is even. 

\begin{lem}
\label{lem-galois-reflection-via-flop-A-2r-case}
The Beauville-Bogomolov degree $(\tilde{e}_{\bar{r}},\tilde{e}_{\bar{r}})$ 
is negative.
The isomorphism $H^2(X,\Integers)\rightarrow H^{4n-2}(X,\Integers)$,
induced by the Beauville-Bogomolov pairing, maps the class 
${\displaystyle 
\frac{-e_{\bar{r}}}{(e_{\bar{r}},e_{\bar{r}})}=
\frac{-2\tilde{e}_{\bar{r}}}{(\tilde{e}_{\bar{r}},\tilde{e}_{\bar{r}})}
}$ 
to the class
$\tilde{e}_{\bar{r}}^\vee$. Consequently, the reflection
\[
g(x) \ \ := \ \ x-
\frac{2(x,\tilde{e}_{\bar{r}})}
{(\tilde{e}_{\bar{r}},\tilde{e}_{\bar{r}})}\tilde{e}_{\bar{r}}
\]
has integral values, for all $x\in H^2(X,\Integers)$.
The Hodge-isometry $g$ induces an automorphism
of $Def(X)$, which belongs to the Galois group $G$.
\end{lem}

\begin{proof}
We indicate the modifications needed in the proof
of the analogous Lemma \ref{lem-galois-reflection-via-flop}.

\underline{Step 1:}
Choose $C\subset Def(X)$ as in Step 1 of the proof of
Lemma \ref{lem-galois-reflection-via-flop}. 
Let $\widetilde{B}_{\bar{r}}\rightarrow B$ be the double cover 
corresponding to the choice of a line in a fiber of 
$E_{\bar{r}}^0\rightarrow B$.
We get a natural $\PP^1$-bundle 
$\widetilde{E}^0_{\bar{r}}\rightarrow \widetilde{B}_{\bar{r}}$
and a morphism $\eta:\widetilde{E}^0_{\bar{r}}\rightarrow \X^0_C$ 
with an injective differential, and so the normal
bundle $N_\eta$ of the morphism is locally free. 
Step 1 of the proof of
Lemma \ref{lem-galois-reflection-via-flop} goes through, 
with $\PP^1_t$, $t\in \widetilde{B}_{\bar{r}}$. 

\underline{Step 2:}
In Step 2 of the proof of
Lemma \ref{lem-galois-reflection-via-flop}
we defined a subscheme $Z\subset \widetilde{U}\times\widetilde{U}$.
We need to redefine $Z$ as the subscheme 
$\Delta_{\widetilde{U}}\cup [E^0_{\bar{r}}\times_{B}E^0_{\bar{r}}]$.
Note that $Z$ has three irreducible components. 
Let $\overline{Z}\subset X\times X$ be the closure of $Z$.
We get 
$
\int_X\tilde{e}_{\bar{r}}\tilde{e}_{\bar{r}}^\vee=
2\int_X[E_{\bar{r}}][F_{\bar{r}}]=-2,
$
\[
\overline{Z}_*(\tilde{e}_{\bar{r}})=2\overline{Z}_*([E_{\bar{r}}])=
-2[E_{\bar{r}}]=-\tilde{e}_{\bar{r}},
\]
and so the analogues of the results of Step 2 hold.

\underline{Step 3:}
We use Example \ref{example-A-2} and 
Lemma \ref{lemma-correspondence-for-A-2-flop} to construct the flop 
$\X'$ of $\X^0_C$.
Choose a covering $\{W_\alpha\}$ of $\Y^0_C$, open in the analytic topology, 
satisfying the following property:
If $b_\alpha$ is a point of $B_\alpha:=B\cap W_\alpha$, then 
$\pi_1(B_\alpha,b_\alpha)$ acts trivially on the graph of the fiber 
$\pi^{-1}(b_\alpha)$.
Set $\widetilde{W}_\alpha:=\tilde{\nu}^{-1}(W_\alpha)$ 
to obtain an open covering of $\X^0_C$. 
Set $E_\alpha:=E^0_{\bar{r}}\cap\widetilde{W}_\alpha$. Then $E_\alpha$ 
is a reducible divisor in $\widetilde{W}_\alpha$ with two irreducible
components $E_{\alpha,1}$ and $E_{\alpha,2}$.
If $B_\alpha$ is empty, set $\widetilde{W}'_\alpha:=\widetilde{W}_\alpha$.
Otherwise, let $\widetilde{W}'_\alpha$ be the result of the sequence of three 
flops, starting with the flop of $\widetilde{W}_\alpha$ along
$E_{\alpha,1}$, then along $E_{\alpha,2}$, and finally along 
$E_{\alpha,1}$ again. 
Similarly, let $\widetilde{W}''_\alpha$ be the result of the sequence of three 
flops, starting with the flop of $\widetilde{W}_\alpha$ along
$E_{\alpha,2}$, then along $E_{\alpha,1}$, and finally along 
$E_{\alpha,2}$ again. Example \ref{example-A-2} shows that the 
bimeromorphic isomorphism
$\widetilde{W}'_\alpha\rightarrow \widetilde{W}_\alpha\rightarrow 
\widetilde{W}''_\alpha$ is actually biregular. 
Hence, $\widetilde{W}'_\alpha$ is independent of the choice of ordering
of the two irreducible components of $E_\alpha$. We can thus glue 
$\{\widetilde{W}'_\alpha\}$ to a smooth analytic space $\X'$, 
locally isomorphic to $\X^0_C$, admitting 
morphisms $\psi':\X'\rightarrow C$, and
$\tilde{\nu}':\X'\rightarrow \Y^0_C$. 
Let $\widehat{X}\subset \X^0_C\times \X'$ be the closure of the graph 
of the isomorphism
\[
[\X^0_C\setminus\widetilde{U}]\ \ \cong \ \ 
[\Y^0_C\setminus U] \ \ \cong \ \ [\X'\setminus\widetilde{U}].
\]
Set $Z':=\widehat{X}\cap[\widetilde{U}\times \widetilde{U}]$.
Then $Z'$ is generically reduced
along each of its irreducible components, and the reduced induced subscheme of 
$Z'$ is isomorphic to 
$\Delta_{\widetilde{U}}\cup[E^0_{\bar{r}}\times_B E^0_{\bar{r}}]$, 
by Lemma \ref{lemma-correspondence-for-A-2-flop}.
%Let $\overline{Z}\subset [X\times X]$ be the closure of the latter. 
The proof of Step 3 of Lemma \ref{lem-galois-reflection-via-flop}
now goes through to show that 
$\overline{Z}_*:H^2(X,\Integers)\rightarrow H^2(X,\Integers)$
is a Hodge-isometry.

The rest of the proof is identical to that of 
Lemma \ref{lem-galois-reflection-via-flop}.
\end{proof}

%{\bf Proof of Theorem \ref{thm-G-is-isomorphic-to-product-of-weyl-groups}:}
The following lemma completes the proof of 
Theorem \ref{thm-G-is-isomorphic-to-product-of-weyl-groups}.
Let $B\in \B$ and $E^0_{\bar{r}}\subset E_B$ be as above.
We indicate the changes required in the statement and proof of Lemma
\ref{lem-Weyl-group-action}, once we drop
Assumption \ref{assumption-no-orbits-of-two-adjacent-nodes}.

\begin{lem}
\label{lem-Weyl-group-action-A-even}
\begin{enumerate}
\item
\label{lem-item-cartan-matrix-A-even}
The matrix with entry $-\int_X\tilde{e}_i^\vee\tilde{e}_j$
in the $i$-th row and $j$-th column is the Cartan matrix of type 
$B_{\bar{r}}$.
\item
\label{lem-item-rank-A-even} 
The lattices 
$\tilde{\Lambda}_B$ and $\tilde{\Lambda}^\vee_B$ both have rank $r$.
\item
\label{lem-item-dual-root-systems-A-even}
The isomorphism
$H^2(X,\RationalNumbers)\rightarrow H^{4n-2}(X,\RationalNumbers)$,
induced by the Beauville-Bogomolov pairing, restricts to an isometry
from $\tilde{\Lambda}_B\otimes_\Integers\nolinebreak\RationalNumbers$ onto 
$\tilde{\Lambda}_B^\vee\otimes_\Integers\RationalNumbers$,
mapping the class
${\displaystyle \frac{-2\tilde{e}_i}{(\tilde{e}_i,\tilde{e}_i)}}$
to the class $\tilde{e}_i^\vee$. 
\item
\label{lem-item-reflection-group-is-weyl-group-A-even}
Let $W_B$ be the subgroup of $G$ generated by the reflections with
respect to $\tilde{e}_i$, as in Lemmas 
\ref{lem-galois-reflection-via-flop} and 
\ref{lem-galois-reflection-via-flop-A-2r-case}.
Then $W_B$ is isomorphic to the Weyl group of type $B_{\bar{r}}$.
\item
\label{lemma-item-G-equal-W-A-even} 
$G=\prod_{B\in\B}W_B$.
\end{enumerate}
\end{lem}

\begin{proof}
\ref{lem-item-cartan-matrix-A-even}) We have 
$
\tilde{e}_i^\vee(\tilde{e}_{\bar{r}})=2\int_X[E_{\bar{r}}][F_i]=
-2\sum_{f_k\in\Gamma f_{\bar{r}}}(f_k,f_i).
$
We see that $-\tilde{e}_i^\vee(\tilde{e}_{\bar{r}})$ is equal to the 
right hand side of the equation in Lemma \ref{lemma-folding},
in the case $(i,j)=(i,\bar{r})$. 
The equality of $-\tilde{e}_{i}^\vee(\tilde{e}_j)$ with the right hand 
side of the equation in Lemma \ref{lemma-folding}, for $j\neq \bar{r}$,
is proven in part \ref{lem-item-cartan-matrix} of Lemma 
\ref{lem-Weyl-group-action}.
Consequently, the matrix 
$-\tilde{e}_{i}^\vee(\tilde{e}_j)$ is the Cartan matrix of type $B_{\bar{r}}$,
the type of the Dynkin diagram obtained by folding that of $A_r$
(see \cite{carter}, Section 13.3).

\ref{lem-item-rank-A-even}) The proof is identical
to that of Lemma \ref{lem-Weyl-group-action}
part \ref{lem-item-rank}.

\ref{lem-item-dual-root-systems-A-even}) 
Follows from Lemma \ref{lem-galois-reflection-via-flop-A-2r-case},
if $i=\bar{r}$, and from Lemma
\ref{lem-galois-reflection-via-flop}, for $1\leq i\leq \bar{r}-1$.

\ref{lem-item-reflection-group-is-weyl-group-A-even}) 
The proof is identical
to that of Lemma \ref{lem-Weyl-group-action}
part \ref{lem-item-reflection-group-is-weyl-group}.

\ref{lemma-item-G-equal-W-A-even})
Let $G_B$ be the image of $G$ 
in $O(\Lambda_B)$ via the composition
$G\rightarrow G_L\rightarrow O(\Lambda_B)$. 
It suffices to prove the inclusion $G_B\subset W_B$,
for every component $B$ violating Assumption 
\ref{assumption-no-orbits-of-two-adjacent-nodes},
by the reduction argument provided in Lemma
\ref{lem-Weyl-group-action} part \ref{lemma-item-G-equal-W}.

We rephrase first the above results in the language of root systems.
Let $\tilde{\Phi}_B$ be the union of all $W_B$ orbits of
the simple roots $\tilde{e}_i$, $1\leq i \leq \bar{r}$.
Define $\tilde{\Phi}^\vee_B$ similarly, with respect to $\tilde{e}_i^\vee$.
We get that $\tilde{\Lambda}_B$ is a root lattice of
the root system $\tilde{\Phi}_B$  of type $C_{\bar{r}}$, and 
$\tilde{\Lambda}_B^\vee$ is the root lattice of the root system 
 $\tilde{\Phi}^\vee_B$ of type $B_{\bar{r}}$.

We define next $\Lambda_B$ as a root lattice of type $B_{\bar{r}}$.
Set $\bar{\Lambda}_B=\Lambda_B$, with 
a basis of simple roots 
${\rm span}_\Integers\{\bar{e}_1, \dots \bar{e}_{\bar{r}}\}$,
where $\bar{e}_i=e_i$, $1\leq i \leq \bar{r}$.
Set $\bar{e}^\vee_i=[F_i]$, for $1\leq i \leq \bar{r}-1$,
$\bar{e}^\vee_{\bar{r}}=2[F_{\bar{r}}]$, and 
$\bar{\Lambda}_B^\vee:=
{\rm span}_\Integers\{\bar{e}^\vee_1, \dots, \bar{e}^\vee_{\bar{r}}\}$.
We get the root systems $\bar{\Phi}_B\subset \bar{\Lambda}_B$ and
$\bar{\Phi}_B^\vee\subset \bar{\Lambda}_B^\vee$.
Part \ref{lem-item-dual-root-systems} still holds, replacing 
$\tilde{e}_i$ with $\bar{e}_i$.
Thus $\bar{\Lambda}_B$ is the root lattice of the root
system $\bar{\Phi}_B$ of type $B_{\bar{r}}$ and 
$\bar{\Lambda}_B^\vee$  is the root lattice of the root
system $\bar{\Phi}_B^\vee$ of type $C_{\bar{r}}$.
The proof of part \ref{lemma-item-G-equal-W} of Lemma
\ref{lem-Weyl-group-action} applies to the pair of dual root systems
$\bar{\Phi}_B\subset \bar{\Lambda}_B$ and 
$\bar{\Phi}_B^\vee\subset \bar{\Lambda}_B^\vee$, and shows the 
the inclusion $G_B\subset W_B$.
\end{proof}

\begin{rem}
\label{rem-BC}
The union $\tilde{\Phi}_B\cup\bar{\Phi}_B$ in $\Lambda_B$
is a root system of type $BC_{\bar{r}}$, by 
the proof of part \ref{lemma-item-G-equal-W-A-even} above.
Similarly, the union $\tilde{\Phi}_B^\vee\cup\bar{\Phi}_B^\vee$ in 
$\Lambda_B^\vee$ is a root system of type $BC_{\bar{r}}$.
\end{rem}

%**********
% End \preprint
%***********
}

\section{Examples}
\label{sec-examples}
%****************************************************************
%
%****************************************************************
\subsection{Hilbert schemes}
\label{sec-Hilbert-scheme-of-K3-resolving-ADE}
Let $\bar{S}$ be a projective $K3$ surface with  $ADE$-singularities at
$\Q:=\{Q_1,\dots, Q_k\}\subset \bar{S}$, and 
$c:S\rightarrow \bar{S}$ the crepant resolution, so that $S$ is a smooth 
$K3$ surface. Let $S^{[n]}$ be the Hilbert scheme of length $n$ subschemes of 
$S$, $n\geq 2$, and $\pi:S^{[n]}\rightarrow \bar{S}^{(n)}$
the composition, of the Hilbert-Chow morphism
$S^{[n]}\rightarrow S^{(n)}$ onto the symmetric product of $S$, with 
the natural morphism $S^{(n)}\rightarrow \bar{S}^{(n)}$ induced by $c$. 
Set $W_0:=\Integers/2\Integers$ and let $W_i$ be the Weyl group of the
root system of the type of the singularity of $\bar{S}$ at $Q_i$,
$1\leq i \leq k$.
We apply Theorem \ref{thm-G-is-isomorphic-to-product-of-weyl-groups}
%Proposition \ref{prop-G-L-equal-W}
to show that $W:=\prod_{i=0}^k W_i$ is isomorphic to the Galois group $G$ of
$f:Def(S^{[n]})\rightarrow Def(\bar{S}^{(n)})$, introduced in
Lemma \ref{lemma-galois-cover}. 

A point $P$ of $\bar{S}^{(n)}$ 
is a function $P:\bar{S}\rightarrow \Integers_{\geq 0}$, 
with finite support $supp(P)$, satisfying $\sum_{Q\in \bar{S}}P(Q)=n$.
Set \\
$\Sigma:=\{P \ : \ \sum_{i=1}^kP(Q_i)+
\sum_{Q\in supp(P)\setminus \Q} [P(Q)-1]>0\}$, \\
$\Sigma_0:=\{P \ : \ \sum_{i=1}^kP(Q_i)+
\sum_{Q\in supp(P)\setminus \Q} [P(Q)-1]>1\}$.

\noindent
$\Sigma$ is the singular locus of $\bar{S}^{(n)}$ and $\Sigma_0$
is its dissident locus. 
The set $\Sigma\setminus \Sigma_0$ is smooth, with
$2n-2$ dimensional connected components: \\
$B_0:=\{P \ : \ \sum_{i=1}^k P(Q_i)=0 \ \mbox{and} 
\sum_{Q\in supp(P)\setminus \Q} [P(Q)-1]=1\}$,\\
$B_j:=\{P \ : \ P(Q_j)=1 \ \mbox{and} 
\sum_{\stackrel{i=1}{i\neq j}}^k P(Q_i)+
\sum_{Q\in supp(P)\setminus \Q} [P(Q)-1]=0\}$.

$\bar{S}^{(n)}$ has a singularity of type $A_1$ along $B_0$
and the singularity along $B_j$, $j\geq 1$, 
has the same type as that of $\bar{S}$ at 
$Q_j$. Set $U:=\bar{S}^{(n)}\setminus\Sigma_0$ and 
$\widetilde{U}:=\pi^{-1}(U)$.
Let $E$ be the exceptional divisor of $\pi$. The intersection
$E^0:=E\cap\widetilde{U}$ has k+1 connected components. 
One, a $\PP^1$ bundle over $B_0$, is irreducible. 
The connected component of $E^0$ over $B_j$, $j\geq 1$, 
is isomorphic to the product
of the fiber of $S$ over $Q_j\in\bar{S}$ with $B_j$. 
Choose a point $b_j\in B_j$, $0\leq j \leq k$. 
We see that $\pi_1(B_j,b_j)$ acts trivially on the Dynkin graph
of the fiber $\pi^{-1}(b_j)$, $0\leq j \leq k$. 
Theorem \ref{thm-G-is-isomorphic-to-product-of-weyl-groups} 
implies the equality $G=W$.

%****************************************************************
%
%****************************************************************
\subsection{O'Grady's $10$-dimensional example}
\label{sec-ogrady-10-dimensional-example}
Let $S$ be a $K3$ surface, $K(S)$ its topological $K$-group, and $v\in K(S)$
the class of an ideal sheaf of a length
$2$ zero dimensional subscheme of $S$.
There is a system of hyperplanes in the ample cone of $S$, called $v$-walls,
that is countable but locally finite \cite{huybrechts-lehn-book}, Ch. 4C.
An ample class is called {\em $v$-generic}, if it does not
belong to any $v$-wall.
Choose a $v$-generic ample class $H$. Let $M:=M_H(2v)$ be the moduli
space of Gieseker $H$-semi-stable sheaves with class $2v$.
$M$ is singular, but it admits a projective symplectic 
resolution $\beta:X\rightarrow M$ \cite{ogrady-10}.
Let $Y$ be the Ulenbeck-Yau compactification of the 
moduli space of $H$-stable locally free sheaves with class $2v$.
There is a natural morphism
$\phi:M\rightarrow Y$, which is an isomorphism along the locus
of stable locally free sheaves \cite{jun-li}.
Let $\pi:X\rightarrow Y$ be the composition 
\[
X\LongRightArrowOf{\beta} M \LongRightArrowOf{\phi} Y.
\]

\begin{lem}
\label{lemma-Ogrady-Galois-group-is-W-G-2}
\begin{enumerate}
\item
\label{lemma-item-weyl-group-is-G2}
The Galois group of $f:Def(X)\rightarrow Def(Y)$ is the Weyl group of type 
$G_2$. 
\item
\label{lemma-item-root-lattice-as-direct-summand}
\cite{rapagnetta}
The lattice $H^2(X,\Integers)$ is the orthogonal direct sum 
$\Lambda\oplus H^2(U,\Integers)$, where $\Lambda$ is the 
(negative definite) root lattice of type $G_2$, and
$H^2(U,\Integers)$ is defined in Lemma \ref{lemma-cohomology-of-X-versus-Y} 
part \ref{lemma-item-saturated}. 
\end{enumerate}
\end{lem}

%The above statement fits nicely with a result of Rapagnetta stating that the 
%lattice $H^2(X,\Integers)$, endowed with the Beauville-Bogomolov form,
%is isometric to the direct sum of $H^2(S,\Integers)$
%and the root lattice of type $G_2$ \cite{rapagnetta}.
%The bilinear pairing on the latter root lattice is negative definite
%and primitive integral, i.e., the intersection matrix is the multiple
%of the one in Example \ref{example-folding-D-4} by a factor of $-3$.

\begin{proof} 
%(of Lemma \ref{lemma-Ogrady-Galois-group-is-W-G-2})
The proof of Lemma \ref{lemma-Ogrady-Galois-group-is-W-G-2} consists
of two steps.
In step 1 we recall that the singular locus $\Sigma_Y$ of $Y$ 
is irreducible, and $Y$ has $D_4$ singularities along the
complement $B_Y:=\Sigma_Y\setminus \Sigma_{Y,0}$ of the dissident locus. 
This fact is implicitly proven in \cite{ogrady-10}.
I thank A. Rapagnetta and M. Lehn for explaining this fact to me.
In the short step 2 we observe that the image of $\pi_1(B_Y,b)$, $b\in B_Y$, in
the automorphism group of the Dynkin diagram of the fiber $\pi^{-1}(b)$,
is the full automorphism group; i.e., the symmetric group $Sym_3$.
The folded Dynkin diagram is thus of type $G_2$, by Example
\ref{example-folding-D-4}, and
the Galois group is the Weyl group of $G_2$, by 
Theorem \ref{thm-G-is-isomorphic-to-product-of-weyl-groups}.
The orthogonal direct sum decomposition follows,
by Corollary \ref{corollary-dual-root-systems}.

\underline{Step 1:}
While $M=M_H(2v)$, the moduli space
$M_H(v)$ is the Hilbert scheme $S^{[2]}$ of length
$2$ zero dimensional subschemes of $S$.
The singular locus $\Sigma_M$ of $M$ is isomorphic to the symmetric 
square\footnote{Lehn and Sorger showed that $X$ is the blow-up of $M$ 
centered at $\Sigma_M$ \cite{lehn-sorger-blow-up}.
} 
$M(v)^{(2)}$ and the dissident locus $\Sigma_{M,0}$ is the diagonal of
$M(v)^{(2)}$ \cite{ogrady-10}. 
$\Sigma_Y$ is isomorphic to
the fourth symmetric power $S^{(4)}$, and
$\Sigma_{Y,0}$ consists of the big diagonal in $S^{(4)}$
(see \cite{ogrady-10}, paragraph
following the proof of Prop. 3.1.1). 
Note that $B_Y:=\Sigma_Y\setminus \Sigma_{Y,0}$ is connected.

We review first the main ingredients of the proof that $\Sigma_Y=S^{(4)}$. 
Any $H$-semi-stable locally free sheaf of class $2v$ is $H$-stable, 
by \cite{ogrady-10}, Lemma 1.1.5.
$H$-stable sheaves correspond to smooth points of $M$ 
\cite{mukai-symplectic-structure}. 
Hence, the singular locus $\Sigma_M$ is contained in the sublocus 
$D\subset M$ parametrizing equivalence classes of $H$-semi-stable sheaves, 
which are not locally free. Each equivalence class is represented
by a unique isomorphism class of an $H$-poly-stable sheaf. 
%Let $D^{st}\subset D$ be the Zariski open 
%subset parametrizing $H$-stable such non-locally-free sheaves.
The reflecsive hull $E^{**}$, of every $H$-poly-stable sheaf $E$
corresponding to a point in $D$, 
is necessarily isomorphic to 
$\StructureSheaf{S}\oplus\StructureSheaf{S}$
\cite{ogrady-10}, Prop. 3.0.5. 
We get a short exact sequence
\[
0\rightarrow E\rightarrow \StructureSheaf{S}\oplus\StructureSheaf{S}\rightarrow
Q_E\rightarrow 0,
\]
where $Q_E$ is a sheaf of length $4$.
$Q_E$ determines a point of $S^{(4)}$. 
The construction yields a surjective morphism
$\restricted{\phi}{D}:D\rightarrow S^{(4)}$,
which is the restriction of the morphism $\phi$, by the results of Jun Li
\cite{jun-li}. 

We need to recall the proof of the surjectivity of 
$\restricted{\phi}{D}:D\rightarrow S^{(4)}$. It suffices to prove that 
$\restricted{\phi}{D}$ is dominant. 
Let $D^0\subset D$ be the Zariski dense open subset, 
where the support of $Q_E$ consists of four distinct points. 
Let $B_Y$ be the complement in $S^{(4)}$ of the big diagonal. We get that 
$\phi$ restricts to a morphism $\restricted{\phi}{D^0}:D^0\rightarrow B_Y$.
O'Grady proves that the latter is a $\PP^1$-bundle
\cite{ogrady-10}, Prop. 3.0.5. 
Following is the identification of the fiber.
Let $b\subset S$ be a length four subscheme consisting of
four distinct points $(q_1, q_2, q_3, q_4)$. 
A poly-stable sheaf $E$ in $D$, 
with $Q_E$ isomorphic to $\StructureSheaf{b}$, 
consists of a choice of a one dimensional subspace of each of the fibers
$E^{**}_{q_i}$. These four fibers can be identified, since $E^{**}$
is trivial.
Hence, a point $E$  in the fiber $D_b$ of $M\rightarrow Y$ over $b$ 
corresponds to a choice of a 
point $(\ell_1,\ell_2,\ell_3,\ell_4)$ of 
$\PP^1\times\PP^1\times\PP^1\times\PP^1$, modulo
the action of $\Aut(E^{**})$. 
If $\ell_1=\ell_2=\ell_3$, then the ideal sheaf $I_{q_4}$ is a subsheaf
of $E$, contradicting the semi-stability of $E$. Consequently, 
$H$-semi-stability of $E$ implies that each point of $\PP^1$
appears at most twice in $(\ell_1,\ell_2,\ell_3,\ell_4)$. 
The two conditions are in fact equivalent. 
The quotient $D_b$ is thus isomorphic to the GIT quotient 
$\overline{M}_{0,4}$ of $\PP^1\times \PP^1\times\PP^1\times\PP^1$ by the 
diagonal action of $PGL(2)$. Now $\overline{M}_{0,4}$ is isomorphic to $\PP^1$.

We are ready to identify the fiber of $\pi:X\rightarrow Y$ over $b\in B_Y$.
Let $D_b$ be the fiber $\phi^{-1}(b)$. 
Let $E\in D_b$ be the polystable sheaf, such that 
$\ell_i=\ell_j$ and $i\neq j$. Set $Z:=\{q_i,q_j\}$ and $W:=b\setminus Z$.  
We get the inclusion $I_W\subset E$ of the 
ideal sheaf of the length two subscheme of $S$ supported on $W$,
and the quotient $E/I_W$ is the ideal sheaf $I_Z$.
Polystability implies that $E=I_W\oplus I_Z$. 
The intersection $D_b\cap \Sigma_M$ thus consists precisely of three 
points $I_W\oplus I_Z$, where 
$W\cup Z=\{q_1,q_2,q_3,q_4\}$ is a decomposition of $b$ as the union of
two disjoint subsets of cardinality two. 
Set $\widetilde{B}_Y:=D^0\cap \Sigma_M$. Then $\widetilde{B}_Y$ 
is a Zariski dense open subset of $B_M:=\Sigma_M\setminus\Sigma_{M,0}$.
$M$ has $A_1$-singularities
along $\widetilde{B}_Y$, by \cite{ogrady-10}, Proposition 1.4.1. Hence, 
the total transform of $D_b$ in $X$ consists of the union of the strict
transform of $D_b$ and three smooth rational curves over the
three points of intersection $D_b\cap \Sigma_M$. 
The graph, dual to the 
fiber of $\pi:X\rightarrow Y$ over $b$, is thus a Dynkin graph of type $D_4$. 
We conclude that $Y$ has $D_4$ singularities along $B_Y$, by Namikawa's
classification \cite{namikawa-deformations}, section 1.8.

\underline{Step 2:}
We have seen that $\widetilde{B}_Y$ is a connected unramified 
cover of $B_Y$ of degree $3$.
Let $E_D$ be the proper transform of $D$ in $X$ and 
$E_\Sigma$ the exceptional divisor of $X\rightarrow M$. 
Then $E_\Sigma^0\rightarrow B_Y$ is the composition
of a $\PP^1$-bundle $E_\Sigma^0\rightarrow \widetilde{B}_Y$ and the
covering map $\widetilde{B}_Y\rightarrow B_Y$. 
The fundamental group of $B_Y$ is the symmetric group $Sym_4$, and it acts on 
$\widetilde{B}_Y$ via the natural homomorphism $Sym_4\rightarrow Sym_3$,
permuting the three decompositions $b=W\cup Z$ of $b$. 
We conclude that the image $\Gamma$ of $\pi_1(B_Y,b)$, 
in the automorphisms group of the Dynkin graph of type $D_4$, is the full 
automorphism group of the latter.
This completes the proof of Lemma \ref{lemma-Ogrady-Galois-group-is-W-G-2}.
\end{proof}

\begin{rem}
Part \ref{lemma-item-root-lattice-as-direct-summand} of
Lemma \ref{lemma-Ogrady-Galois-group-is-W-G-2} 
falls short of determining the Beauville-Bogomolov pairing
on $H^2(X,\Integers)$, which is the main result of 
\cite{rapagnetta}. Missing still is the determination of 
the rank of $H^2(U,\Integers)$ and of two constants
$\lambda_1, \lambda_2$ defined below.
The bilinear pairing on the root lattice $\Lambda$ of $G_2$
is determined only up to a constant. Rapagnetta proved that 
the matrix of the bilinear pairing on $\Lambda$ is the multiple
of the one in Example \ref{example-folding-D-4} by a factor of $\lambda_1=-3$.
Rapagnetta furthermore shows that 
the direct summand $H^2(U,\Integers)$ is Hodge-isometric to 
$H^2(S,\Integers)$. 
The proof of the latter statement goes as follows:
Donaldson's homomorphism $\mu:H^2(S,\RationalNumbers)
\rightarrow H^2(Y,\RationalNumbers)$ is injective,
by \cite{ogrady-10}, Claim 5.2. 
Compose with restriction to $U$ to get an injective homomorphism 
$\tilde{\mu}:H^2(S,\RationalNumbers)
\rightarrow H^2(U,\RationalNumbers)$. The homomorphism $\tilde{\mu}$
is surjective, since $b_2(X)=24$, by \cite{rapagnetta}, Theorem 1.1.
There exists a rational number $\lambda_2$, 
such that $\lambda_2\tilde{\mu}$ is an isometric embedding, 
with respect to the intersection pairing on
$S$ and the Beauville-Bogomolov pairing induced on the
direct summand $H^2(U,\Integers)$ of $H^2(X,\Integers)$. The existence
of $\lambda_2$ follows from the calculation of
the degree $10$ Donaldson polynomial in 
\cite{ogrady-10}, equation (5.1). 
A short argument, using the polarization of that Donaldson polynomial,
shows that $\lambda_2=1$ 
(see the first part of the proof of \cite{rapagnetta}, Theorem 3.1). 
The unimodularity, of the $K3$ lattice, shows that
$\pi^*\left(\mu[H^2(S,\Integers)]\right)$ is saturated in $H^2(X,\Integers)$,
and hence equal to $H^2(U,\Integers)$.
%*******
%\hide
%*******
\hide{
The equality $\lambda_1=-3$ may be restated more conceptually
as follows. Mukai endowed the free abelian group $K(S)$ 
with a unimodular bilinear pairing isometric to the
orthogonal direct sum of $H^2(S,\Integers)$ and the rank $2$ hyperbolic 
lattice. Let $v^\perp\subset K(S)$ be  the orthogonal complement 
of the class $v$. 
The Donaldson homomorphism $\mu$ extends to the Mukai homomorphism
(??? there does not exists a semi-universal sheaf ???)
$\tilde{\mu}:v^\perp\rightarrow H^2(\M(2v),\RationalNumbers)$.
The equality $\lambda_1=-3$ follows from the equality $\lambda_1=1$
and the statement that 
$\beta^*\circ \tilde{\mu}$ is an isometric embedding. 
%*******
% End \hide
%*******
}
\end{rem}

%*******************************************************************
%
%*******************************************************************
\subsection{Contractions of moduli spaces of sheaves on a $K3$}
I thank K. Yoshioka for kindly explaining to me the following examples.
%The above is a simple construction of a divisorial contraction
%of a smooth and projective moduli space of stable sheaves on
%a $K3$ surface $S$. 
Fix an ample line bundle $H$ on a $K3$ surface $S$.
The notion of Gieseker $H$-stability admits a refinement, due to
Matsuki and Wentworth, called $w$-twisted $H$-stability,
where $\omega$ is a class in $\Pic(S)\otimes_\Integers\RationalNumbers$
\cite{matsuki-wentworth}.
Yoshioka extended this notion for sheaves of pure one-dimensional support
\cite{yoshioka-twisted-stability-pure-one-dimension}.
The moduli space $\M^\omega_H(r,c_1,c_2)$,
of $S$-equivalence classes of $\omega$-twisted $H$-semistable sheaves with
rank $r$ and Chern classes $c_i$, is a projective scheme
\cite{matsuki-wentworth,yoshioka-twisted-stability-pure-one-dimension}.
There is a countable, but  locally finite, collection of walls, defining  
a chamber structure on 
$\Pic(S)\otimes_\Integers\RationalNumbers$, once
we fix $H$, $r$, and $c_i$, $i=1,2$.
A class $\omega$ is called {\em generic}, if it does not lie on a wall.

When $\omega$ is generic, it is often the case that
$\omega$-twisted $H$-semistability, for sheaves with the specified rank and 
Chern classes, is equivalent to $\omega$-twisted $H$-stability.
The moduli space $\M^\omega_H(r,c_1,c_2)$ is smooth and holomorphic
symplectic, for such a generic $\omega$ \cite{mukai-symplectic-structure}. 
The moduli space is in fact deformation equivalent to $S^{[n]}$ and is thus
a projective irreducible holomorphic symplectic manifold 
\cite{yoshioka-abelian-surface}. 
If we deform $\omega$ to a class $\bar{\omega}$ in the closure of
the chamber of $\omega$, then there exists a morphism 
$\pi^{\omega}_{\bar{\omega}}:
M^\omega_H(r,c_1,c_2)\rightarrow M^{\bar{\omega}}_H(r,c_1,c_2)$
\cite{matsuki-wentworth}.
One can carefully analyze the exceptional locus of the 
contraction in many examples. This is done for two dimensional moduli spaces
in \cite{onishi-yoshioka}. They obtain many examples 
where the morphism $\pi^{\omega}_{\bar{\omega}}$ is the crepant resolution
of a normal $K3$ surface with ADE singularities. 
These methods can be extended to the higher dimensional examples 
considered in 
\cite{yoshioka-lie-algebra}, Example 3.2 and Proposition 4.3,
to obtain divisorial contractions with Weyl groups of irreducible
root systems of ADE-type. 
A detailed explanation would, regrettably, take us too far afield.

\hide{
We explain a special case of 
\cite{yoshioka-lie-algebra}, Proposition 4.3.
Consider a projective $K3$ surface $\bar{S}$, with a single simple 
singularity. Let $c:S\rightarrow \bar{S}$ be its crepant resolution, 
so that $S$ is a smooth $K3$ surface. 
Let $C_1$, \dots, $C_r$, be the irreducible components of the exceptional locus
ordered, so that the intersection $(C_i,C_j)$ is a Cartan matrix of type $ADE$.
Choose a primitive class $\xi\in \Pic(S)$, such that 
$(\xi,[C_i])=0$, for all $i$. Let $v\in K(S)$
be the class of rank $0$, with $c_1(v)=\xi$, and $\chi(v)=1$.
Given a sheaf $F$,  
set $\chi_\omega(F):=\chi()$
For a generic $\omega$, the moduli space $M^\omega_H$
}

%*****************************************************************
%
%*****************************************************************
\section{Generalizations}
\label{sec-generalizations}
%*****************************************************************
%
%*****************************************************************
\subsection{Symplectic generalizations}
We thank Y. Namikawa for pointing out the following generalizations
of Lemma \ref{lemma-galois-cover},
allowing the singular symplectic $Y$ to be affine, or considering
only partial resolutions of $Y$.

\subsubsection{Partial resolutions}
Let $Y$ be a projective symplectic variety. 
A $\RationalNumbers$-factorial terminalization of $Y$ is a morphism 
$\pi:X\rightarrow Y$ from a projective $\RationalNumbers$-factorial 
symplectic variety $X$ with only terminal singularities, 
which is an isomorphism over the 
non-singular locus $Y_{reg}$ of $Y$.
$Y$ need not admit a symplectic resolution \cite{KLS}, but it always 
admits a $\RationalNumbers$-factorial terminalization, 
by \cite{namikawa-Q}, Remark 2, and 
recent results in the minimal model program \cite{BCHM}.
The Hodge structure of $H^2(X_{reg},\ComplexNumbers)$ is pure, 
any flat deformation of $X$ is locally trivial, i.e., 
it preserves all the singularities of $X$ \cite{namikawa-Q},
the Kuranishi space $Def(X)$ is smooth, and the Local Torelli Theorem
holds for $X$ \cite{namikawa-extension-of-2-forms}, Theorem 8.
The analogue of Theorem \ref{thm-namikawa} above for a symplectic
$\RationalNumbers$-factorial terminalization $\pi:X\rightarrow Y$
was proven in 
\cite{namikawa-Q}, Theorem 1. The analogue of Lemma \ref{lemma-galois-cover}
above extends, by exactly the same argument.
We expect Theorem \ref{thm-G-is-isomorphic-to-product-of-weyl-groups}
to generalize as well.

%***************************************************
%
%***************************************************
\subsubsection{Affine examples}
\label{sec-affine-examples}
%The second one is, to replace $Y$ by an affine
%symplectic variety with a C^*-action with positive
%weight and $X$ by a symplectic resolution of $Y$
%(eg. a nilpotent orbit closure and its Springer
%resolution)
%In this case, the usual deformation functor
%$Def$ should be replaced by the "Poisson deformation functor".
Theorem \ref{thm-namikawa} was further extended by Namikawa to the 
the case where 
$\pi:X\rightarrow Y$
is a crepant resolution of an affine symplectic variety $Y$
with a good $\ComplexNumbers^*$-action and with a positively weighted 
Poisson structure 
(\cite{namikawa-poisson-deformations-of-affine}, Theorems 2.3 and  2.4). 
A $\ComplexNumbers^*$-action on an affine variety $Y$ is {\em good},
if there is a closed point $y_0\in Y$ fixed by the action,
and $\ComplexNumbers^*$ acts on the maximal ideal of $y_0$
with positive weights. 
Examples include Springer resolutions of the closure of a nilpotent orbit
of a complex simple Lie algebra. More generally, Nakajima's quiver varieties
fit into the above set-up \cite{nakajima-reflections}. 

The analogue of Lemma \ref{lemma-galois-cover}
above extends. 
Infinitesimal Poisson deformations of $X$ are governed by 
$H^2(X,\ComplexNumbers)$, by 
\cite{namikawa-flops-and-poisson-deformations}, Corollary 10. 
The only change to the proof of Lemma \ref{lemma-galois-cover}
is that we replace
the period domain by $H^2(X,\ComplexNumbers)$, on which the
monodromy group acts, and we do not need to consider any pairing on
$H^2(X,\ComplexNumbers)$. The Kuranishi Poisson deformation space
$PDef(X)$ is an open analytic neighborhood of $0$ in 
$H^2(X,\ComplexNumbers)$, and 
the period map is replaced by the inclusion
$PDef(X) \subset H^2(X,\ComplexNumbers)$.
The analogue of Theorem 
\ref{thm-G-is-isomorphic-to-product-of-weyl-groups} for  quiver varieties
follows from the results in \cite{nakajima-reflections}.

%*****************************************************************
%
%*****************************************************************
\subsection{Modular Galois covers in the absence of simultaneous resolutions}
\label{sec-absence-of-simultaneous-resolutions}
Let $Y$ be a normal compact complex variety,
$X$ a connected compact K\"{a}hler manifold, and $\pi:X\rightarrow Y$ a 
morphism, which is bimeromorphic. Assume that 
$R^1_{\pi_*}\StructureSheaf{X}=0$.
Then $\pi$ induces a morphism $f:Def(X)\rightarrow Def(Y)$ and
$\pi$ deforms as a morphism $\nu$ of the semi-universal families, 
fitting in a commutative diagram (\ref{diagram-f-general}),
by \cite{kollar-mori}, Proposition 11.4. 
Let $M\subset Def(Y)$ be the image of $f$.
We keep the notation of diagram (\ref{diagram-f-general}) and 
impose the following additional assumptions:
%**************
% Assumption
%**************
\begin{assumption}
\label{assumption-for-CY-crepant-resolution}
\begin{enumerate}
\item
\label{assumption-def-X-reduced-and-irred}
$Def(X)$ is smooth.
%reduced, irreducible, and normal.
\item
\label{assumption-f-is-proper-and-finite}
The morphism $f$ is proper with finite fibers. 
\item
\label{assumption-local-torelli}
The Local Torelli Theorem holds for $Def(X)$ in the following sense.
There exists a positive integer $i$, such that the local system 
$R^i_{\psi_*}(\RationalNumbers)$ is a 
polarized\footnote{A canonical polarization exists, for example, if
$i$ is the middle cohomology, 
or if $i=2$ and $X$ is an irreducible holomorphic
symplectic manifold.  
} 
variation of Hodge structures, with period domain $\Omega$
and period\footnote{See the reference \cite{griffiths}
for the definitions of variations of Hodge structures 
and period maps.
} 
map $p:Def(X)\rightarrow \Omega$. Furthermore, 
the differential $dp_0$ at $0\in Def(X)$ is injective.
%an embedding of $Def(X)$
%as a closed analytic subvariety of an open subset of  the period domain 
%$\Omega$. 
\item
\label{assumption-pull-back-of-quotient-local-system}
There exists a closed analytic proper subset 
$\Delta\subset M$, such that if we set $U:=M\setminus \Delta$
and $V:=Def(X)\setminus f^{-1}(\Delta)$, then 
the restriction 
$\restricted{\bar{\psi}}{U}:\restricted{\Y}{U}\rightarrow U$ 
is a topological fibration. Furthermore, for every $t\in U$ and 
for every two points $t_1,t_2$ in $V$,
satisfying $f(t_1)=f(t_2)=t$, the pull-back homomorphisms
$\nu_{t_j}^*:H^i(Y_t,\RationalNumbers)\rightarrow 
H^i(X_{t_j},\RationalNumbers)$,
$j=1,2$, are surjective and have the same kernel, which we may denote by $L_t$.
%If $t_1,t_2$ are two points in $V$,
%satisfying $f(t_1)=f(t_2)=t$, 
%and $\nu_{t_j}^*:H^i(Y_t,\Integers)\rightarrow H^i(X_{t_j},\Integers)$,
%$j=1,2$, are the corresponding pull-back homomorphisms, then
%the above assumption asserts that $\ker(\nu_{t_1}^*)=\ker(\nu_{t_2}^*)=L_t$
%and $\nu_{t_j}$ 
%factors through an isomorphsim from 
%$H^i(Y_t,\Integers)/L_t$ onto $H^i(X_{t_j},\Integers)$.
\end{enumerate}
\end{assumption}

Part \ref{assumption-pull-back-of-quotient-local-system} of the
assumption implies that the homomorphism of local systems
\[
\nu^* \ : \ 
f^*\left([R^i_{\bar{\psi}_*}\RationalNumbers\restricted{]}{U}\right)
 \ \ \ \longrightarrow \ \ \ 
R^i_{\psi_*}(\RationalNumbers\restricted{)}{V}
\]
is surjective and its kernel is the pullback $f^*L$ of a local subsystem 
$L\subset [R^i_{\bar{\psi}_*}\RationalNumbers\restricted{]}{U}$
over $U$. Consequently,
the local system $R^i_{\psi_*}(\RationalNumbers\restricted{)}{V}$
is isomorphic to the pullback of the quotient local system 
$[R^i_{\bar{\psi}_*}\RationalNumbers\restricted{]}{U}/L$ over $U$. 
This isomorphism is a generalization of the isomorphism
(\ref{eq-pull-back-by-nu-is-an-isomorphism}) used in the proof of 
Lemma \ref{lemma-galois-cover}.
The image $M$ of $f$ is a closed irreducible analytic subset of $Def(Y)$, by 
part \ref{assumption-f-is-proper-and-finite} of the assumption. 
We endow $M$ with the reduced structure sheaf, and let
$\widetilde{M}\rightarrow M$ be its normalization. 
The morphism $f$ factors through $\widetilde{M}$, since $Def(X)$
is smooth. 
%and we set $\tilde{0}$ to be the image in $\widetilde{M}$ of $0\in Def(X)$. 
The proof of Lemma
\ref{lemma-galois-cover} is easily seen to generalize yielding 
the following result.

\begin{lem} 
\label{lemma-generalized-galois-cover}
Under the above assumptions,
there exist neighborhoods  of $0$ in $Def(X)$ and of $\bar{0}$ in 
$M$, which we denote also by $Def(X)$ and $M$, and
a finite group $G$, acting faithfully on $Def(X)$, 
such that the morphism $f$ is the composition of the
quotient $Def(X)\rightarrow Def(X)/G$, an 
isomorphism $Def(X)/G\rightarrow \widetilde{M}$, 
and the normalization $\widetilde{M}\rightarrow M$.
%which is an isomorphism over the non-singular locus of $M$. 
The group $G$ acts on
$H^*(X,\Integers)$ via monodromy operators preserving the Hodge structure.
\end{lem}

%***********************************************
%
%***********************************************
\subsection{Calabi-Yau threefolds with a curve of ADE singularities}
We provide examples of the set-up of section
\ref{sec-absence-of-simultaneous-resolutions}. 
Let $\pi:X\rightarrow Y$ be a crepant resolution, by a Calabi-Yau 
threefold $X$, of a normal projective variety $Y$ with a uniform
ADE singularity along a smooth curve $C$ of genus $g\geq 1$. 
The morphism $\pi$ is assumed to restrict to an isomorphism over 
$Y\setminus C$.

\begin{lem}
\label{lemma-CY-threefolds}
Assume\footnote
{The assumption is automatically satisfied if $g=1$, since
$Y_t$ is in fact smooth. In the $g=1$ case 
$Y$ is deformation equivalent to $X$, and $Def(X)$ and $Def(Y)$
have the same dimension, by \cite{wilson}, Proposition 4.4 and 
the erratum of \cite{wilson}. 
In that case, $M=Def(Y)$.
If $g>1$, then $Y$ is not deformation equivalent to $X$,
and the expected dimension of $Def(Y)$ is larger.
%If the contraction $\pi:X\rightarrow Y$
%is {\em primitive}, i.e., does not factor through an intermediate
%normal variety, then $Y$ is smoothable by deformations
%\cite{gross-primitive}, Theorem 1.3.
The variety $Y_{t}$ would thus be singular, if $g>1$.
} 
that for a generic $t\in M$, $Y_t$ is normal with only isolated singular 
points, and $\nu_t:X_t\rightarrow Y_t$ is a small resolution.
Then the conclusion of Lemma
\ref{lemma-generalized-galois-cover} holds.
\end{lem}

\begin{proof}
\footnote{I thank T. Pantev for explaining to me the results of
\cite{donagi-pantev-ADE}, leading to the formulation of this Lemma.}
%I thank B. Hassett, Y. Namikawa, and  T. Pantev, for helpful suggestions
%leading to the proof of the Lemma.
We verify the conditions of Assumption
\ref{assumption-for-CY-crepant-resolution}.
$Def(X)$ is smooth, by the Bogomolov-Tian-Todorov Theorem
\cite{Bo,Ti,To}. 
Condition \ref{assumption-def-X-reduced-and-irred} follows.
The Local Torelli Condition 
\ref{assumption-local-torelli} (with $i=\dim_\ComplexNumbers(X)$) 
is well known for K\"{a}hler manifolds
with trivial canonical bundle.
Condition \ref{assumption-f-is-proper-and-finite} follows from the
assumption that the singularities along $C$ are of $ADE$-type, by the 
same argument used in the proof of Theorem 
\ref{thm-namikawa}
(see \cite{namikawa-deformations}, Claim 2 in the proof of Theorem 2.2).
The key is the uniqueness of the crepant resolution of $ADE$-singularities.

It remains to prove 
Condition \ref{assumption-pull-back-of-quotient-local-system}.
We endow $M$ with the reduced structure sheaf, 
so that $M$ is reduced and irreducible.
Let $\Delta\subset M$ be the union of the singular locus of $M$,
the branch locus of $f$, the locus where
$Y_t$ has singularities, which are not isolated singular points, and
the locus where $Y_t$ has more than the number of 
isolated singular points of $Y_{t'}$, of a given topological type,
for a generic $t'\in M$. 
Set $U:=M\setminus \Delta$. 
$U$ is non-empty, by assumption, open, and 
$\Y$ clearly restricts to a topological fibration over $U$. 
Set $V:=f^{-1}(U)$. 
Fix $s\in V$ and set $t:=f(s)$. 
The surjectivity of 
$\nu_s^*:H^3(Y_t,\Integers)\rightarrow H^3(X_s,\Integers)$ follows from the 
following general fact.

\begin{claim}
Let $N$ be a three dimensional normal complex variety, smooth away from
a finite set of isolated singular points, admitting  a small resolution 
$\phi:M\rightarrow N$.
Then $\phi^*:H^3(N,\Integers)\rightarrow H^3(M,\Integers)$ 
is surjective. 
\end{claim}

\begin{proof}\footnote{I thank Y. Namikawa for kindly providing this proof.}
Associated to $\phi$ 
we have the canonical Leray filtration $L$ on $H^q(M,\Integers)$, 
and the Leray spectral sequence $E_r^{p,q}$ converging to 
$H^{p+q}(M,\Integers)$, with
$E_2^{p,q}=H^p(N,R_{\phi_*}^q\Integers)$ and 
$E_\infty^{p,q}=Gr_L^pH^{p+q}(M,\Integers)$
(see \cite{voisin-book-vol2}, Theorem 4.11).

Denote by $d_r^{p,q}:E_r^{p,q}\rightarrow E_r^{p+r,q-r+1}$ the
differential at the $r$ term. 
Note the equality $E_2^{3,0}=H^3(N,\Integers)$,
since $N$ is normal and thus the fibers of $\phi$ are connected. 
We have a surjective homomorphism 
\[
H^3(N,\Integers)=E_2^{3,0}\rightarrow E_\infty^{3,0}=Gr_L^3H^3(M,\Integers),
\]
since $d_r^{3,0}:E_r^{3,0}\rightarrow E_r^{3+r,1-r}$
vanishes for $r\geq 2$.
Hence, it suffices to prove that $E^{p,q}_\infty=Gr^p_LH^3(M,\Integers)$
vanishes, for $p\neq 3$. 
Now $E^{p,q}_\infty$ is a sub-quotient of
$E^{p,q}_2$, so it suffices to prove the vanishing of 
$E^{2,1}_2=H^2(N,R^1_{\phi_*}\Integers)$, 
$E^{1,2}_2=H^1(N,R^2_{\phi_*}\Integers)$, and 
$E^{0,3}_2=H^0(N,R^3_{\phi_*}\Integers)$.
The first two vanish, since for $q>0$ the sheaf 
$R^q_{\phi_*}\Integers$ is supported on the zero dimensional set of isolated
singular points, and so $H^p(Y,R^q_{\phi_*}\Integers)$ vanishes, for $p>0$.
$E^{0,3}_2$ vanishes, since the sheaf $R^3_{\phi_*}\Integers$
vanishes, because the fibers of $\phi$ have real dimension at most $2$.
\end{proof}

Fix $t\in U$ and 
let $s_1$, $s_2$ be two points of $V$ with $f(s_1)=f(s_2)=t$.
We prove\footnote{I thank B. Hassett for kindly providing this argument.} 
next that $\ker(\nu_{s_1}^*)=\ker(\nu_{s_2}^*)$.
Let $\overline{X}$ be the closure 
in $X_{s_1}\times X_{s_2}$ of the graph of the birational isomorphism
$X_{s_1} \rightarrow Y_t\leftarrow X_{s_2}$. 
Choose a resolution of 
singularities $\tilde{X}\rightarrow \overline{X}$.
Let $p_i:\tilde{X}\rightarrow X_{s_i}$ be the projection.
We have the equality $\nu_{s_1}\circ p_1=\nu_{s_2}\circ p_2$. 
Set $\mu:=\nu_{s_i}\circ p_i$. 
Let $[\tilde{X}]\in H_6(\tilde{X},\RationalNumbers)$ be the fundamental class.
The homomorphisms 
$p_i^*:H^3(X_{s_i},\RationalNumbers)\rightarrow 
H^3(\tilde{X},\RationalNumbers)$ satisfy
\[
p_{i_*}\left([\tilde{X}]\cap p_i^*\alpha\right)=
p_{i_*}([\tilde{X}])\cap\alpha=
[X_{s_i}]\cap\alpha,
\] for all $\alpha\in H^3(X_{s_i},\RationalNumbers)$,
by the projection formula.
The homomorphisms $p_i^*$ are thus injective, since the cap product 
$[X_{s_i}]\cap:H^3(X_{s_i},\RationalNumbers)\rightarrow 
H_3(X_{s_i},\RationalNumbers)$ is an isomorphism, by Poincare Duality.
Hence, $\ker(\nu_{s_i}^*)=\ker(\mu^*)$, $i=1,2$.
This completes the proof of Lemma
\ref{lemma-CY-threefolds}.
\end{proof}

\begin{example}
The following simple example is discusses in \cite{donagi-pantev-A1}, section 
3.2, in connection with large $N$ duality for Dijkgraaf-Vafa geometric
transitions in string theory.
Consider the complete intersection of a quadric Q and a quartic F hypersurfaces
in $\PP^5$. Let $\widetilde{Q}$ be the quadratic form with zero locus $Q$.
If $\widetilde{Q}$ has rank $3$, so that $Q$ has an $A_1$-singularity along a
plane $P$, then for a generic $F$, the complete intersection
$Y:=Q\cap F$ has an $A_1$-singularity along the smooth 
plane quartic $C:=P\cap F$. The blow-up $\pi:X\rightarrow Y$,
of $Y$ along $C$, yields a Calabi-Yau threefold $X$.
If the quadratic form $\widetilde{Q}_t$ has rank $4$, so that the quadric $Q_t$
is singular along a line $\ell\subset Q$, then for a generic $F$, 
the complete intersection $Y_t:=Q_t\cap F$ has four ordinary double points
along $\ell\cap F$. The proper transforms of $Y_t$, 
in each of the two small resolutions of $Q_t$, is a crepant resolution 
$X_{s_i}\rightarrow Y_t$, $i=1,2$. 
Both $X_{s_i}$ are deformation equivalent to $X$.
The two corresponding points $s_i\in Def(X)$, $i=1,2$, 
are interchanged by the Galois involution
of $f:Def(X)\rightarrow M$.
In this case $Def(X)$ and $M$ are $86$-dimensional, while
$Def(Y)$ is $89$-dimensional. For a generic pair of
smooth quadric $Q_t$ and quartic $F$, the complete intersection
$Q\cap F$ is a smooth Calabi-Yau threefold, which is 
deformation equivalent to $Y$ but not to $X$.
\end{example}

Additional examples of contractions, of the type considered in 
Lemma \ref{lemma-CY-threefolds}, with $A_n$ singularities along 
a curve, are provided in \cite{katz-morrison-plesser}.
An example of type $C_2$ singularities along 
a curve is provided in 
\cite{szendroi-enhanced}, Example 8. 
Szendr\Hun{o}i states Lemma 
\ref{lemma-CY-threefolds}, as well as the analogue of Theorem 
\ref{thm-G-is-isomorphic-to-product-of-weyl-groups} in this set-up,
as a meta-principle in \cite{szendroi-enhanced}, Proposition 6. 
He proceeds to prove that proposition 
in a beautiful family of examples, where $X$ is a rank two vector 
bundle over $C$
%with determinant $\omega_C^{-1}$, 
and $Y$ is its quotient by a finite group of 
vector bundle automorphisms, which acts on each fiber as
a subgroup of $SL(2)$ 
(see \cite{szendroi-artin-group}, Proposition 2.9 part iii).
Szendr\Hun{o}i's main emphasis is the construction of a 
braid group action on the level of derived categories,
which he carries out in general
\cite{szendroi-enhanced,szendroi-artin-group}.

%****************************************************************
% Bibliography:
%****************************************************************
%\journal{\begin{footnotesize}}

%\journal{\end{footnotesize}}

\end{document}